\documentclass[12pt]{amsart}%
\usepackage{amsfonts}
\usepackage{epsfig}
\usepackage{graphicx}
\usepackage{amsmath}
\usepackage{amsfonts}
\usepackage{amssymb}
\usepackage{latexsym}
\usepackage{rotating}
\usepackage{pstricks, pst-node, pst-text, pst-3d}
\usepackage{amsbsy}
\usepackage{bm}
\usepackage{dsfont}

\input{amssym.def}
\newtheorem{theorem}{Theorem}
\newtheorem{proposition}{Proposition}
\newtheorem{corollary}{Corollary}
\newtheorem{lemma}{Lemma}
\newtheorem{remark}{Remark}

\newtheorem{definition}{Definition}
\theoremstyle{remark}
\newtheorem{example}{Example}

\newcommand{\Id}{\text{\rm Id}}

\newcommand{\ad}{\text{\rm ad}}
\newcommand{\spn}{\text{\rm span}\,}

\newcommand{\End}{\text{\rm End}}

\newcommand{\dist}{\text{\rm dist}}
\newcommand{\capac}{\text{\rm cap}}

\newcommand{\g}{\mathfrak{g}}

\newcommand{\R}{\mathbb{R}}

\newcommand{\heis}{\mathbb{H}}

\begin{document}
\title[Extremal functions for modules...]{Extremal functions for modules of systems of measures}
\author[M.~Brakalova, I.~Markina,  and A.~Vasil'ev]{Melkana Brakalova$^{\dag}$, Irina Markina$^{\ddag}$, and Alexander Vasil'ev$^{\ddag}$}

\thanks{The author $^{\dag}$ was partially supported by a Faculty Research Grant, Fordham University, USA, and by the Norwegian Research Council grant \#204726/V30.  The authors$^{\ddag}$ were partially supported by
the grants of the Norwegian Research Council \#204726/V30 and  \#213440/BG, as well as by EU FP7 IRSES program STREVCOMS, grant  no. PIRSES-GA-2013-612669}
 

\address{M.~Brakalova: Department of Mathematics, Fordham University, 407 John Mulcahy Hall Bronx, NY 10458-5165, USA}
\email{brakalova@fordham.edu}

\address{I.~Markina and A.~Vasil'ev: Department of Mathematics, University of Bergen, P.O.~Box 7800,
Bergen N-5020, Norway}
\email{irina.markina@uib.no}
\email{alexander.vasiliev@math.uib.no}

\begin{abstract}
We study  Fuglede's $p$-module of systems of measures in condensers in  Euclidean spaces and on polarizable Carnot groups.
We apply and generalize a result by Rodin, which  provides an explicit method for finding the extremal function and the 2-module of a foliated family of curves  in $\mathbb R^2$, to  a variety of settings. In the planar case, we apply Rodin's method to obtain  estimates for the conformal module of a parallelogram and of a ring domain using directional dilatations.  In $\mathbb R^n,$ we {identify the extremal function and} compute the $p$-module of images of  families of connecting curves and  of separating sets with respect to the plates of a condenser  under  homeomorphisms of certain regularity.  Then we calculate the module and find the extremal measures for the spherical ring domain on polarizable Carnot groups and extend Rodin's theorem to the spherical ring domain on the Heisenberg group. 
\end{abstract}
\maketitle

\tableofcontents


\section{Introduction}\label{introduction}


Let $\Omega$ be a bounded domain in a polarizable Carnot group $G$, a particular case of which is the Euclidean space $\mathbb R^n,$ $n\geq 2$, and let $D_0$ and $D_1$ be two disjoint compacts in the closure $\overline{\Omega}$ of $\Omega$. The triple $(\Omega; D_0,D_1)$  is called a {\it condenser} in $G$, and the compacts $D_0$ and $D_1$ -- its plates. Two important quantities related to the condenser $(\Omega; D_0,D_1)$ are the module
of the family of  curves  $\Gamma(\Omega; D_0,D_1)$ connecting the compacts $D_0$ and $D_1$ in $\Omega$, and the module of  the surfaces, or  more generally, the sets $\Sigma(\Omega; D_0,D_1)$ separating $D_0$ and $D_1$ in $\Omega$. Observe, that the module of  $\Gamma(\Omega; D_0,D_1)$ is closely related to the notion of capacity of $D_0$ and $D_1$ relatively to $\Omega$, see~e.g.,~\cite{AikOht, Geh62, Hesse,  markina2003, Shlyk90, Shlyk93, zim68, zim69}. Since the notion of Lipschitz surfaces is  somewhat restrictive for  the development of the theory of modules and related topics,  it is convenient to use  the notion of modules of systems of measures  introduced by Fuglede~\cite{Fug} in 1957.

We  first recall the definition.
Let $(X,\mathfrak{M},m)$ be an abstract measure space with a fixed  measure
$m:\mathfrak{M}\to[0,+\infty]$ defined on the $\sigma$-algebra $\mathfrak{M}$ of  subsets in $X$.
We denote by $\mathcal{M}$ the system of all measures $\mu$ in $X$, whose
domains of definition contain $\mathfrak{M}$. 
With an arbitrary system of measures $E\subset\mathcal{M}$ we associate the class of non-negative measurable functions $\rho$
defined in $X$ and satisfying the condition
\begin{equation}\label{admissfun}
 \int_{X}\rho\,d\mu\geqslant 1,\qquad \mu\in E.
\end{equation}
 We call such $\rho$ {\it admissible}, and
we write $\rho\wedge\mu$ if~\eqref{admissfun} holds for a measure $\mu,$ and $\rho\wedge E$ if~ \eqref{admissfun} holds for every $\mu\in E$.
\begin{definition}\label{FugledeM}
 For $0<p<\infty$, the $p$-module $M_p(E)$ of  the system of measures E is defined as
\[
M_p(E)=\inf\limits_{\rho\wedge E}\int_{X}\rho^p\,dm,
\]
interpreted as $+\infty$ if the set $\{\rho\colon \rho\wedge E\}$ is empty. A function $\rho_0,$ $\rho_0\wedge E$, is called {\it extremal} for the $p$-module  $M_p(E)$  if $M_p(E)=\int_{X}\rho_0^p\,dm$. 
\end{definition}

Definition~\ref{FugledeM} is a natural generalization of the concept of the module of a family of curves in $\mathbb R^n$, $n\geq 2$. Given a family
$\Gamma$ of locally rectifiable curves in  space $X=\mathbb R^n$, one can regard  $\mathfrak{M}$ as the Borel $\sigma$-algebra, $m$ as the Lebesgue $n$-dimensional measure, and $\mu$ as  the arc-length of a curve which is a measure associated with this curve. This construction for the case of $\mathbb R^n$ was carefully
developed in detail in~\cite[Chapter 2]{Ohtsuka} and also in~\cite{Fug}.

We call a system of measures $E_0\subseteq E$ {\it extremal} for $M_p(E)$ if
 $$\int_{X}\rho_0\,d\mu=1\quad \text{and} \quad M_p(E)=M_p(E_0)=\int_{X}\rho_0\,d\mu,
 $$
 for $M_p$-almost all $\mu\in E_0$. If, in addition, $E_0=E$ for $M_p$-almost all $\mu,$ we call $E_0$ a  {\it complete  extremal system of measures}.
The family of {\it extremal measures} does not always exist and is not, in general, uniquely determined, although it is in the cases we consider. We consider this in more detail in the special case of {\it extremal family  of curves}  in Section~\ref{planar}.

The above definition  of an extremal function and measure is closely related to Beurling's Criterion, a nice and straightforward sufficient condition,
which guarantees that an admissible function for a family of curves in $\mathbb R^2$ is extremal for its module. 
Badger \cite{Badger}  generalized  Beurling's criterion to $\mathbb R^n$  making it necessary and sufficient for  the Fuglede $p$-module of measure systems.

It is well known  that in  $\mathbb R^n,$ $n\geq 2,$ the module of the family of all rectifiable curves $\Gamma,$ in  a spherical ring domain $\Omega=R_{ab},$ connecting  the two boundary  concentric spheres of radii $a,b$, $0<a<b<\infty,$  is equal to the module of the family of radial curves $\Gamma_0,$  connecting these boundary spheres, and that $\Gamma_0$ is extremal.    In the same spirit, the  family of concentric spheres of radii $r, \ a<r<b,$ separating the boundary spheres is extremal for the module of all Lipschitz separating surfaces in $R_{ab}$,~\cite{Geh62, Shlyk90, va61, zim66}. The extremal functions are also known.
Finding the  extremal function  and the $p$- module is, in general,  quite a difficult task, possible to be completed only  in a few cases. One of the main aims of this paper is extending a relatively less known result by Rodin~\cite{Rodin}, which provides a method for finding the extremal function that leads to an explicit calculation of the module of a complete extremal family of curves in the plane, to $\mathbb R^n, \ n\geq 2,$ and to polarizable Carnot groups.

In Section \ref{planar} we study  applications of 2-modules of families of curves in $\mathbb R^2$, introduce Rodin's theorem, Beurling's criterion,  the notion of extremal curve family, and  briefly  discuss its existence and uniqueness for curves. We apply Rodin's method to obtain exact formulas for the extremal function and the module of parallel slanted 
intervals in a parallelogram,  and of logarithmic spirals  in a ring domain, as well as of  their images.  We give  estimates for the  module of a quadrilateral and of a ring domain in terms of directional dilatations.

In Section \ref{RodinS},  we provide extensions of Rodin's theorem \cite{Rodin} to higher dimensions for rather general type of condensers. One of the main results is Theorem~\ref{Rodin1} which explicitly calculates the $p$-module $M_p(\Gamma_0^\prime)$ of an extremal family of curves $\Gamma_0^\prime=f(\Gamma_0)$, where $f$ is a diffeomorphism in $\Omega$ and $\Gamma_0=\Gamma_0(\Omega; D_0,D_1)$. An analogous result for the module of a system of measures associated with sets separating $D_0$ and $D_1$ in $\Omega$ is obtained in Theorem \ref{Rodin2}.
The tools to handle this more difficult case are developed in Section \ref{coareasection}, where we consider special cases of the coarea and change-of-variable formulas. 
Two typical examples of $(\Omega;D_0,D_1)$ are
 a {\it cylinder}  and a {\it spherical ring domain} in $\mathbb R^n$.
 Using monotonicity property of modules with respect to families of curves we come to some useful inequalities.
Turning from the smooth case to more general settings,  we discuss several minimal regularity properties of the homeomorphism $f$. In particular, we can assume $f$ to be Sobolev regular and of finite distortion.




After a thorough analysis has been developed on Carnot groups, see e.g., \cite{Heinonen95, HH, Vodopis90},  analogous problems can be formulated for these groups. The only locally rectifiable curves on Carnot groups are so-called horizontal curves. To define a ring domain on Carnot groups one can use an analogue of the Euclidean norm which is a homogeneous functions with respect to the  anisotropic dilation $\delta_s$, $s>0$, respecting grading of the corresponding Lie algebras. Unfortunately, the radial curves $\delta_{(\cdot)}$ (considered as functions of $s$) are not horizontal in general. On the Heisenberg group, Kor\'anyi and Reimann~\cite{KorReim87}  found another family of `radial' curves which are horizontal and orthogonal (in a correct sense) to the spheres defined as the level sets of the homogeneous norm. It was also shown that this family of curves is extremal for the ring domain. The existence of a homogeneous norm and corresponding horizontal family of radial curves was proved for some other classes of  Carnot groups in~\cite{balogh2001}, which received the name {\it polarizable Carnot groups}. One of our aims in Section \ref{groups} is to show that the family of concentric spheres separating two boundaries of the spherical ring domain is also extremal for the module of  all separating sets in this ring on polarizable Carnot groups. We extend  Theorem~\ref{Rodin1} to this geometric setting.

\section{ Rodin's theorem and extremality in $\mathbb R^2$}\label{planar}

The modules  of a quadrilateral and of a ring domain, were originally introduced  and studied in the works of Gr\"otzsch, Teichm\"uller,  Ahfors,   and many others \cite{ Ahlfors, Grotz, Jen, LehtoVirt, Teich}  in the first half of the 20$^{\rm th}$ century.  It is natural to assume that these concepts have provided some of  the inspiration behind the development of the notion of  the extremal length of a family of curves by Beurling and published  in a  joint work  with Ahlfors  \cite{AhlB} in 1950. The module of a  family of curves is defined as the reciprocal of the extremal length of that family. These  concepts and their generalizations have become  a  powerful tool in the study of a wide range of function-theoretic properties of domains in the complex plane, space, and on general  metric  measure spaces.

\subsection{Extremal functions and extremal families of curves}

First, we reformulate Definition~\ref{FugledeM}  for the special case of 2-module of a family of curves  in $\mathbb R^2,$ state Beurling's criterion, discuss extremal functions, and introduce the notion of extremal families of curves.

\begin{definition}\label{moduledef}Let $\Omega$ be a domain in $\mathbb R^2$, and let  $\rho\colon \mathbb R^2\to \mathbb R$ be a
non-negative, measurable function. Let $\Gamma$ be a family of locally rectifiable curves in $\Omega$. We say that $\rho$ is an {\it admissible} function for $\Gamma$ if 
\begin{equation}\label{adm1}
\int_{\gamma}\rho \, ds\geq 1,\quad \text{for any $\gamma\in\Gamma$,}
\end{equation}
where $\gamma$ is parametrized by arc-length $s$. 
Then the quantity
\begin{equation}\label{mod1}
M_2(\Gamma):=\inf\left\{\int_{\Omega}\rho^2 (x) dm\colon \text{among admissible $\rho$}\right\},
\end{equation}
 where  $m$ is the Lebesgue measure in $\mathbb R^2,$ is called the 2-module of $\Gamma.$ If the equality in~\eqref{mod1} is attained for some function $\rho_0$, then this function is called {\it extremal} for the   module problem $M_2(\Gamma)$.
\end{definition}

The extremal function is essentially unique when it exists (see, e.g., \cite{Garmar}), and it exists when we possibly  exclude some subfamilies of curves of vanishing module, see~\cite{Fug} or Proposition~\ref{modulusprop}, item~(7). 

 In what follows, we also use {\it conformal invariance, monotonicity and subadditivity } of the 2-module, properties derived in e.g.  \cite{Ahlfors, Garmar, Pom, Vas}.


\begin{definition} Let $\rho_0$ be an extremal function for $M_2(\Gamma)$. We say that  $\Gamma_0\subset\Gamma$  is an extremal family of curves  for  $M_2(\Gamma)$ if it satisfies the following two conditions:

\begin{equation}\label{extremal1}
\int_{\gamma}\rho \, ds=1,\quad \text{for $M_2$- almost all $\gamma\in\Gamma_0$, and $M_2(\Gamma)=M_2(\Gamma_0).$}
\end{equation}
\end{definition}

Below we  discuss the existence and uniqueness of extremal families, up to a family of curves of vanishing $2$-module .

Let $Q_{ab}=\{ (x,t): 0\leq x\leq a, 0 \leq t\leq b\}$ and let $D_0=\{(x,t): 0\leq x\leq a, t=0\}$, $D_1=\{(x,t): 0\leq x\leq a, t=b\}$.
Let $\Gamma=\Gamma(Q_{ab},D_0,D_1)$ be the family of all curves connecting $D_0$ and $D_1$. One can show that the function $\rho_0=\dfrac 1b $ is extremal for $M_2(\Gamma)$, and therefore $M_2(\Gamma)=\dfrac ab,$ see e.g.,  \cite{Ahlfors, LehtoVirt}.   Also, the  family $\Gamma_0\subset \Gamma$ of vertical segments connecting  $D_0$ and $D_1$ satisfies the properties (\ref{extremal1}) and therefore is extremal for  $M_2(\Gamma)$. The extremal family for the module  $M_2(\Gamma)$  in $Q_{ab}$, is unique up to a  family of curves of vanishing $2$-module.  The family $\Gamma\setminus \Gamma_0,$ has  $\rho_0$ as an extremal function, too.  However, there is no extremal family for $M_2(\Gamma\setminus \Gamma_0).$  Consider any conformal mapping $f\colon Q_{ab}\rightarrow Q.$  The image  $f(\Gamma_0)$  is an extremal family for the $2$-module $M_2(\Gamma(Q;f(D_0),f(D_1)))$  of the family of curves connecting $f(D_0)$ and $f(D_1)$ in $Q$, and the extremal function is $\rho=\frac{1}{b}|f'|$.

Consider the square $Q_{11}$,  and the family of all curves $\Gamma^\star$ connecting the horizontal sides or the vertical sides of $Q_{11}$. Let  $\Gamma_0$ be the family of vertical segments connecting the horizontal sides,   and $\Sigma_0$ be the family of horizontal segments connecting the vertical sides of $Q_{11}$, (see \cite{Badger}).  The extremal function for $M_2(\Gamma^\star)$ is $\rho_0=1$, and each of  $\Gamma_0$, $\Sigma_0$, and $\Gamma_0\cup\Sigma_0$ is an  extremal family for  $M_2(\Gamma^\star).$ In this example the extremal family for $M_2(\Gamma^\star)$ is not uniquely determined (up to a set of curves of vanishing $2$-module). In this paper we consider  extremal family of curves that are  unique up to a set of curves of $2$-module equal to 0.
  
 Below we state Beurling's extremaility criterion, formulated by Ahlfors, \cite[Theorem~4-4]{Ahlfors} and cited as Beurling's unpublished work.  It provides a sufficient condition
guaranteeing that an admissible function for a family of curves in the plane is extremal for its module.

\medskip
\noindent
{\bf Theorem A}~(Beurling's Extremality Criterion).\label{Beurling}
{\it  Let $\Gamma$ be a family of curves in a domain $\Omega\subseteq \mathbb R^2$ and let $\rho_0$ be an admissible function for $\Gamma$. Suppose that there exists  a subfamily $\Gamma_0$ in $\Omega$, such that
\begin{itemize}
\item $\Gamma_0\subseteq \Gamma$;
\item $\int_\gamma\rho_0\, ds=1$ for every $\gamma\in \Gamma_0$;
\item For all real valued functions $g\in L^2(\Omega), $  the condition $\int_\gamma g\, ds\geq 0$ for all $\gamma \in \Gamma_0,$  implies
\begin{equation}\label{rho}
\int_{\Omega}\rho_0 g \,dm\geq 0.
\end{equation}
\end{itemize}
Then $\rho_0$ is   extremal for  $M_2(\Gamma)$ and
$$M_2(\Gamma_0)=M_2(\Gamma)=\int_{\Omega}\rho_0^2\,dm.$$
}

\medskip
\noindent

\begin{remark} \label{extremalt}
As one can see from the proof of Theorem A \cite{Ahlfors}, one needs to consider only the case $g=\rho-\rho_0$ for all admissible $\rho$ for $M_2(\Gamma)$ to claim extremality of $\rho_0.$ 
Indeed, (\ref{rho}) implies $\int_\Omega \rho_0^2dm \leq \int_\Omega\rho \rho_0dm.$ Squaring both sides, applying the Cauchy-Schwarz inequality
$\left(\int_\Omega \rho_0^2 dm \right)^2  \leq \int_\Omega\rho^2 dm \int_\Omega \rho_0^2 dm,$  implies 
$
\int_\Omega \rho_0^2 dm\leq \int_\Omega \rho^2 dm,
$
for all admissible $\rho.$ Therefore $M_2(\Gamma)=\int_\Omega \rho_0^2 dm$  and $\rho_0$ is extremal.
\end{remark}

In view of this criterion, $\Gamma_0$ is an {\it extremal family of curves} for $M_2(\Gamma)$.  Let us observe that the practical verification of admissibility of $\rho_0$ for the entire family $\Gamma$ is perhaps the most difficult part of this criterion although the other parts are equally important. As a supporting argument, consider the following example.

\begin{example}\label{e4} 
Let $\Omega=Q_{1,h\sin\theta}=[0,1]\times[0,h\sin\theta]\subset \mathbb C$, $\theta\in (0,\pi/2]$, and let $\Omega'=f(\Omega)$, where
 $$f(x,t)=(x+t\cot\theta,t),\quad t\in[0,h\sin\theta],\quad x\in [0,1],\quad h>0,\quad \theta\in (0,\pi/2).$$
Then $\Omega'$ is the parallelogram in $\mathbb C$ with the vertices at $0, 1, 1+h e^{i\theta}, h e^{i\theta}$. Define the  family of locally rectifiable curves $\Sigma$  in $\Omega$ joining the opposite vertical sides of $\Omega$ and by $\Sigma_0$ the subfamily of horizontal intervals.  Let  $\Sigma'=f(\Sigma)$ and $\Sigma'_0=f(\Sigma_0)$.
One has $M_2(\Sigma'_0)=M_2(\Sigma_0)=M_2(\Sigma)=h\sin \theta$, see  \cite[Lemma, Page 35]{Ahlfors2}. At the same time, see~\cite{AndVuor},
\[
M_2(\Sigma')=\frac{\mathbf{K}'}{\mathbf K}(r_{\theta/\pi}),
\]
where
\[
r_{\theta/\pi}=\mu^{-1}\left(\frac{\pi h}{2\sin\theta}\right),\quad \mu(r)=\frac{\pi}{2\sin\theta}\frac{\mathbf{F}\left(\frac{\theta}{\pi},1-\frac{\theta}{\pi};1;1-r^2\right)}{\mathbf{F}\left(\frac{\theta}{\pi},1-\frac{\theta}{\pi};1;r^2\right)}.
\]
Here $\mathbf{K}$ and $\mathbf{K}'$ are complete elliptic integrals, and $\mathbf{F}$ means the Gaussian hypergeometric function $\ _2F_1$.
In this example,  we see that $M_2(\Sigma'_0)$ is rather simple expression, whereas
the calculation of the module of the larger family of curves $\Sigma'$ is  a much harder task and requires explicit conformal maps based on the Weierstrass $\wp$-function.

Due to the monotonicity property we can estimate the module of a family of curves using a convenient subset. For example, the inequality $M_2(\Sigma'_0)\leq M_2(\Sigma')$ gives an interesting lower estimate for the elliptic integrals $$h\sin \theta\leq \frac{\mathbf{K}'}{\mathbf K}(r_{\theta/\pi}). $$ 
\end{example}
  
The modules for $\Sigma'_0$ and $\Sigma'$ coincide only if $\theta=\pi/2$, and in this case $\Sigma_0$ is the extremal  family for  $\Sigma$. In view of Theorem~A, this means that the function $\rho_0(x)\equiv 1$ is extremal for $\Sigma_0$, but is not admissible for $\Sigma$ whereas other conditions of Beurling's criterion remain true.




\subsection{Rodin's theorem and applications}

 First, we state Rodin's theorem which provides an explicit method for calculating the extremal functions and the module of a family of curves.  Later, we obtain upper and lower estimates for the conformal module of quadrilaterals and doubly connected domains,   in terms of directional dilatations, see Definition \ref{directional}.

Let $f$ be a smooth, orientation preserving homeomorphism of $Q_{1b}$ onto a region $Q \in \mathbb R^{2}$, such that the Jacobi matrix  exists and its determinant $J_f$ is strictly positive. Let $\Gamma_0$
be the family of vertical intervals $v_x(t)=\{(x,t)\colon t\in [0,b]; \,\text{$x\in [0,1]$ is fixed}\}$, and let  $c_x(t)=f(v_x(t))\in Q$. Thus, the image of $\Gamma_0$ is  $f(\Gamma_0)=\{c_x\colon [0,b]\to Q,\,x\in [0,1]\}$.

\medskip
\noindent
{\bf Theorem B}~(Rodin's Theorem, \cite{Rodin}). \label{Rodin}{\it Let
\[
\ell(x)=\int_0^b\frac{|\dot{c}_x|^2}{J_f}\,dt,\quad x\in [0,1],
\]
where $\dot{c}_x=\frac{\partial}{\partial t}c_x(t)$.
Then 
\begin{equation}\label{Rodinextremal}
\rho_0(y)=\frac{1}{\ell(x)}\left(\frac{|\dot{c}_x|}{J_f}\right)\circ f^{-1}(y),\quad (x,t)\in Q_{1b},\quad y=f(x,t)\in Q,
\end{equation}
is the extremal function for the $2$-module  of the family $f(\Gamma_0)$ and
\begin{equation} \label{Rodinformulat}
M_{2}(f(\Gamma_0))=\int_{Q}\rho^2_0(y)\,dy=\int_0^1 \ell^{-1}dx.
\end{equation}
}

\medskip
\noindent

To prove Theorem B, one shows that $\rho_0$ in (\ref{Rodinextremal}) satisfies conditions of Theorem~A.  The computation of $M_2(f(\Gamma_0))=\int_\Omega \rho_0^2 \,dm$  leads to  (\ref{Rodinformulat}).

The above result  was  used by Rodin and Warschawski  to characterize the boundary behavior of conformal maps, see, e.g., \cite{RodinW1, RodinW2}.

If a quadrilateral $Q$ is a conformal image of the rectangle $Q_{1b}$, and if $\Gamma=\Gamma(Q_{1b};D_0,D_1)$, then $M_2(f(\Gamma))=M_2(\Gamma)=1/b$.
If $\Sigma$ is the family of curves separating $D_0$ and $D_1$ in  $Q_{1b}$, then $M_2(f(\Sigma))M_2(f(\Gamma))=M_2(\Sigma)M_2(\Gamma)=1$, see e.g., \cite{Ahlfors, Jen, LehtoVirt, Vas}.
Sometimes,  families $\Gamma$ and $\Sigma$ are called {\it conjugate}. 

Let us return back to Example~\ref{e4} of the parallelogram $\Omega'$. As it was observed, the expression  for $M_2(\Sigma')$ is difficult and involves some implicit functions and elliptic integrals. Rodin's theorem allows to give some explicit estimates of  $M_2(\Sigma')$. Denote by $\Gamma'_{\theta}$ the family of parallel slanted intervals connecting the horizontal sides of $\Omega'$. By $\Gamma'$, we denote the whole family of curves connecting the horizontal sides of $\Omega'$, and if  $\theta=\pi/2$, then $\Gamma'_{\theta}=\Gamma_0$. As it was remarked, $M_2(\Sigma'_0)=M_2(\Sigma_0)=h\sin\theta$ and $M_2(\Sigma')\geq h\sin\theta$ because of  monotonicity~of~$M_2$.

\begin{proposition}\label{inequalityt}
\begin{equation}\label{parallelogram1}
M_2(\Gamma'_{\theta})=\frac{\sin\theta}{h}, \quad M_2(\Gamma'_{\theta})M_2(\Sigma'_0)=\sin^2\theta,
\end{equation}
and 
\begin{equation}\label{parallelogram2}
|M_2(\Sigma')-M_2(\Sigma'_0)|\leq h\dfrac{\cos^2\theta}{\sin\theta}.
\end{equation}
\end{proposition}
\begin{proof} 
Applying Theorem~B we have
$$
M_2(\Gamma'_{\theta})=\int_0^1\dfrac{dx}{\int_0^{h\sin \theta} \dfrac{|f_t|^2}{J_f} dt}=\frac{\sin\theta}{h},\ \    M_2(\Sigma'_0)=\int_0^{h\sin \theta}\dfrac{dy}{\int_0^1 \dfrac{|f_x|^2}{J_f} dx}=h\sin\theta, 
$$
because $|f_t|^2=1+\cot^2\theta,\ J_f=1$ and $|f_x|^2=1.$ This implies (\ref{parallelogram1}).
By the {\it monotonicity property} of the module of a family of curves, we have

$$
M_2(\Sigma'_0)\leq M_2(\Sigma') \leq  (M_2(\Gamma'_{\theta}))^{-1}.
$$
Since  $M_2(\Sigma'_0)=h\sin\theta$, using (\ref{parallelogram1})  we obtain  (\ref{parallelogram2}) .
\end{proof}

\begin{remark}
Of course, the shear map in Example~1 is a quasiconformal homeomorphism with the maximal real dilatation
\[
K(\theta)=1+\frac{1}{2}\cot^2\theta+\frac{1}{2}\cot\theta\sqrt{4+\cot^2\theta}\geq 1.
\]
So the estimate $M_2(\Sigma')\leq Kh\sin\theta$ holds trivially. However, $K(\theta)\geq 1/\sin^2\theta$ and the estimate in Proposition~\ref{parallelogram2} is better for this map.
\end{remark}

Applying Rodin's method, we obtain useful estimates for the  conformal modules of quadrilaterals and ring domains by using the monotonicity property of 2-modules of curves.

The vertices of $Q_{ab}$ are usually assigned a specific order, namely $(0,0)$, $(a,0)$, $(a,b)$ and $(b,0)$, and the segments $D_0$ and $D_1$ are called its $a$-sides, thus these are the sides connecting the first and the second and the third and the fourth vertices.  By a {\it quadrilateral}  $Q$ we understand a conformal image $f(Q_{ab})$ of a rectangle $Q_{ab}$ with a fixed order of vertices, the images of the vertices of $Q_{ab}$ and thus with identified $a$-sides, the images of $D_0$ and $D_1.$ The {\it conformal module of the quadrilateral $Q$} with the above assigned order of vertices and $a$-sides is defined to be 
$$M(Q)=b/a.$$ 
Thus, $M(Q)=M_2(\Sigma)=M(Q_{ab})=b/a$, where $\Sigma=\Sigma(Q_{ab};D_0,D_1)$ is the family of curves separating $D_0$ and $D_1$ in $Q_{ab}$. It is well-known that  $M(Q)$ can be obtained by minimizing the Dirichlet integral $\min \iint\limits_Q|\nabla v|^2 dxdy,$ for a set of admissible for the capacity functions, see Definition~\ref{def:p_capac}, and is equal to the capacity of the condenser  determined by the domain $Q$ with the plates which are the images of the vertical sides of $Q_{ab}$.
 
 The parallelogram $P(\varepsilon)$ with vertices $0, 1, 1+\varepsilon+ib, \varepsilon+ib$  is the image of $Q_{1b}$  under the sheer transformation $f(x,t)=(x+\varepsilon t,t). $  The properties of its conformal module, $M(P_\varepsilon),$ have been extensively studied in many papers. For example, Reich~\cite{Reich1} showed that $M(P_\varepsilon)$ is a convex non-decreasing function of $\varepsilon, $ using Steiner symmetrization. This result was extended in a paper of Dubinin and Vuorinen \cite{DubVuo07} who studied the change of conformal modules of polygonal quadrilaterals under some transformations. Properties of conformal modules of polygonal quadrilaterals, including parallelograms, have been investigated using hypergeometric functions and numerical methods in  \cite{AndVuor, HakVuor, HaRaVuo11, RaVuo06} and many others, see the references therein. 
 
Applying Rodin's result, we obtain  the following inequality that provides an estimate for the rate of convergence of $M(P_\varepsilon)$ to $M(Q_{1b}),$ as $\varepsilon \rightarrow 0.$  
\begin{proposition}
 $$|M(P(\varepsilon))-M(Q_{1b})|\leq \dfrac{\varepsilon^2}b.$$
\end{proposition}
The {\it proof} is done in analogy to Proposition \ref{inequalityt}.

 Now we turn to ring domains.
A {\it ring domain} $R$ is a conformal image of the annulus $R_{ab}=\{z\colon a\leq |z|\leq b\}$. The {\it conformal module}  $M(R)$ of $R$  is defined as 
 $$
 M(R)=\dfrac{\log  b/a}{2\pi}.
 $$
Extremal properties of modules of doubly connected regions, or ring domains, have been originally studied in the works of Teichm\"uller and Gr\"otzsch, see \cite{LehtoVirt} for some results and history of the topic. The concept   has played an essential role in the development of the theory of plane quasiconformal mappings and is naturally connected to the $2$-module of families of curves.

 If $\Gamma$ is a family of curves connecting the boundary circles $S_a$ and $S_b$, let $\Sigma$ be the family of curves separating $S_a$ and $S_b$ in  $R_{ab}.$ Denote by  $f(\Gamma)$ and $f(\Sigma)$  their images under a conformal map $f$ that maps $R_{ab}$ onto $R$. One can show that
\[
M_2(f(\Gamma))=M_2(\Gamma)=\frac{2\pi}{\log\, b/a},
\]
and that
\[
M_2(f(\Sigma))M_2(f(\Gamma))=M_2(\Sigma)M_2(\Gamma)=1.
\]
Therefore, $M(R)=M_2(\Sigma)=M_2(f(\Sigma)).$ Thus, the conformal module can be defined as the $2$-module of the family of curves $\Sigma=\Sigma(R_{ab};S_a,S_b)$ separating  the boundaries $S_a$ and $S_b$ in the annulus $R_{ab}$.

We denote by $\Gamma_0$  the family of radial intervals connecting $S_a$ and $S_b$ and by $\Sigma_0$ the family of concentric circles separating $S_a$ and $S_b$. The families of curves $\Sigma_0, f(\Sigma_0), \Gamma_0, f(\Gamma_0,)$ are  extremal families for the corresponding module problems.

Now, let   $f\colon R_{1b}\rightarrow R',$  be an orientation preserving homeomorphism (not necessarily conformal) defined in a neighborhood of  $R_{1b}$.  A point in $R_{1b}$ is called {\it regular} following~\cite{Cazacu1}, if $f$ is differentiable at this point and the Jacobian $J_f$ is  strictly positive.
At a regular point $z=re^{i\theta} \in R_{1b},$ the {\it complex dilatation} of $f$ is $\mu_f=\dfrac{f_{\bar z}}{f_z},$  and the Jacobian is $J_f=|f_z|^2-|f_{\bar z}|^2=|f_z|^2(1-|\mu_f|^2).$

Below we introduce the {\it directional dilatation} of $f$ in the direction $\alpha,$ at a regular point. This concept was studied by Andreian-Cazacu \cite{Cazacu1} for the purpose of generalizing the class of $K$--quasiconformal mappings. 

\begin{definition}\label{directional}
The directional dilatation of $f$  in the direction $\alpha \in \mathbb R,$ is defined~as $$D_{f,\alpha}=\dfrac{|1+e^{-2i\alpha}\mu_f|^2}{1-|\mu_f|^2}.$$ \end{definition}

Denote by $f_r=\dfrac{\partial}{\partial r}f(re^{i\theta})$ and  $f_{\theta}=\dfrac{\partial}{\partial \theta}f(re^{i\theta}).$ Then $f_r= e^{i\theta}(f_z+e^{-2i\theta}f_{\bar z}) $ and $f_\theta= ire^{i\theta}(f_z-e^{-2i\theta}f_{\bar z}).$ Hence,  $D_{f,\theta}=\dfrac{|f_r|^2}{J_f},\ \ \ D_{f,\theta+\frac{\pi}2}=\dfrac{|f_\theta|^2}{r^2J_f}.$ We assume that $0\leq \theta< 2\pi.$

Let $\Gamma_0'$, $\Gamma'$, $\Sigma_0'$, and $\Sigma'$ be the images of  $\Gamma_0$, $\Gamma$, $\Sigma_0$, and $\Sigma$ under $f,$ respectively.
Below we state results concerning modules of  images under $f$ of radial segments and concentric circles in an annulus and the corresponding extremal functions. For simplicity we assume that $f$ is an orientation preserving $C^1$-smooth homeomorphism but the results can be extended, after careful justification, to mappings of exponentially integrable finite distortion and $\mu$-homeomorphisms, which are locally absolutely continuous, see \cite{Brakalova2}.

\begin{proposition}\label{modulesigmaprime} 
The 2-module of the family of curves $\Gamma_0'$ is calculated as
\begin{equation}\label{modulesegments}
M_2(\Gamma_0^\prime)= \int_{0}^{2\pi} \left( \int_1^b \dfrac{D_{f,\theta}}{r} dr  \right)^{-1} d\theta
\end{equation}
The extremal function for $M_2(\Gamma_0^\prime)$ is given by
$$
\rho_0=\dfrac{\dfrac{D_{f,\theta}}{r}}{|f_r|\int_1^b \dfrac{D_{f,\theta}}{r}dr}\circ f^{-1}.
$$
\end{proposition}
\begin{proof}
For $0\leq\theta<2\pi,$ denote by $\gamma_\theta$ the radial segment connecting the concentric circles $|z|=1$ and $|z|=b$, making angle $\theta$ with the $x$-axis, and by $\gamma_\theta^\prime$ its image under $f$. We show that  $\rho_0$ satisfies the conditions of Theorem~A. Indeed, $\int_{\gamma_\theta'}\rho_0\,ds=1$.
Next, let $\rho$ be any admissible function for $M_2(\Gamma_0^\prime).$ Then  $\int_{\gamma_\theta'}\left( \rho-\rho_0 \right) ds \geq 0,$ and
therefore, $\int_1^b \left( \rho-\rho_0 \right)\circ f |f_r| dr \geq 0.$ 
Hence,
$$
\int_{R^\prime} \rho_0(\rho-\rho_0) \,dm=\int_0^{2\pi} \left( \int_1^b \dfrac{D_{f,\theta}}{r} dr  \right)^{-1}\int_1^b (\rho-\rho_0)\circ f |f_r| \,dr d\theta\geq 0.
$$
By Theorem~A,  $M_2(\Gamma_0^\prime)=\int_{R^\prime}\rho_0^2\,dm.$
A simple calculation leads to (\ref{modulesegments}).
\end{proof}

The theorem below was proved in \cite{Brakalova2}. 
\begin{proposition}\label{modulegammaprime} 
The 2-module of the family of curves $\Sigma_0^\prime$ is calculated as
$$M_2(\Sigma_0^\prime)= \int_{1}^{b} \left( \int_0^{2\pi} D_{f,\theta+\frac{\pi}2} d\theta  \right)^{-1} \dfrac{dr}r,$$
and the extremal function is $$\rho_0=\dfrac{D_{f,\theta+\frac{\pi}2}}{|f_\theta|\int_0^{2\pi}D_{f,\theta+\frac{\pi}2}}\circ f^{-1}.$$
\end{proposition}

From  the monotonicity property, Propositions \ref{modulesigmaprime} and \ref{modulegammaprime}, we have  $M_2(\Sigma_0^\prime)\leq M_2(\Sigma^\prime)\leq \left(M_2(\Gamma_0^\prime)\right)^{-1}$. Denote by $M(R^\prime)$ the conformal module of $R'.$ Since $M(R^\prime)=M_2(\Sigma^\prime),$ this leads to the following useful estimate for $M(R^\prime),$ in terms of directional dilatations.

\begin{theorem}\label{upperlower} For the conformal module $M(R^\prime)$ of $R'$  we have
$$
\int_{1}^{b} \left( \int_0^{2\pi} D_{f,\theta+\frac{\pi}2} d\theta  \right)^{-1} \dfrac{dr}r\leq M(R^\prime)\leq \left( \int_{0}^{2\pi} \left( \int_1^b \dfrac{D_{f,\theta}}{r} dr  \right)^{-1} d\theta \right)^{-1}.
$$
\end{theorem}

Observe that the upper and lower estimates are  conformally invariant. Applying the Cauchy-Schwarz inequality leads to
\begin{equation}\label{ringdomainestimaes}
\int_{1}^{b}\left(\int_{0}^{2\pi} D_{f,\theta+\frac{\pi}2} d\theta \right)^{-1}\dfrac{dr}{r}\leq M(R^\prime)\leq \dfrac1{(2\pi)^2} \int_{R_{1b}} \dfrac{D_{f,\theta}}{|z|^2} \, dm.
\end{equation}

Such estimates have been used in e.g., \cite{BrakJenk1, BrakJenk2, Brakalova1, Brakalova2, GuMa03, Lehto73} and others. To obtain (\ref{ringdomainestimaes}) directly, one usually applies the length-area method, which involves  integrating over curves and using Fubini's theorem, see Volkovyskii \cite{Vo50}, Reich and Walczak \cite{Reich}, Gutlyanskii and Martio \cite{GuMa03} and others  (the history and  equivalence of such estimates  under the assumption that $f$ is quasiconformal, have been discussed in \cite{Brakalova0}). 

Now we derive expressions for the modules of a family of logarithmic spirals and their images under a smooth homeomorphism $f$ in terms of directional dilatations.
Let $\beta\in \mathbb R.$  Consider $f\colon R_{1b}\rightarrow R_{1b},$ $f(re^{i\theta})=|z|e^{i(-\beta\log |z|+\theta)},$ which maps the radial segments in $R_{1b}$  into portions of logarithmic spirals in $R_{1b}$ of inclination $\beta$. Denote by $\Gamma_{\beta}=\bigcup_{\theta\in [0,2\pi)} \{z: z=r e^{i(-\beta\log r +\theta)}, 1<r<b\}$ and by  $\Gamma'_{\beta}$ its $f$-image.  Applying previous results and using Rodin's theorem, we obtain the following useful for the applications result.

\begin{proposition}
The $2$-module of the family of logarithmic spirals $\Gamma_{\beta}$ is 
\begin{equation}\label{logspirals}
M_2(\Gamma_{\beta})=\dfrac{2\pi}{(1+\beta^2)\log b.}
\end{equation}
Also 
$$
M_2(\Gamma'_{\beta})=\int_0^{2\pi}\left(  \int_1^b (1+\beta^2) D_{f,\theta_0} \dfrac{dr}{r}\right)^{-1}d\theta,
$$
where $\theta_0=-\beta\log r+\theta+\arctan\beta$.
\end{proposition}
\begin{proof}
Let $z=re^{i\theta}$ represent the polar coordinates of $z$ in the annulus $R_{1b}$. Then $f_r=e^{i(-\beta\log r+\theta)} (1-i\beta).$ Thus $|g_r|^2=(1+\beta^2).$ In addition $J_f=1.$ From Proposition \ref{modulesigmaprime}, it follows that the extremal function for  $M_2(\Gamma_{\beta})$ is

$$\rho_0\circ f(re^{i\theta})=\dfrac{\sqrt{1+\beta^2}}{r\int_1^b(1+\beta^2)\dfrac{dr}{r}}=\dfrac1{\sqrt{1+\beta^2}\left( r\log b\right)}.$$

Thus $$M_2(\Gamma_{\beta})=\int_{R_{1b}} \rho_0^2 dm=\int_1^b \int_0^{2\pi} \dfrac1{1+\beta^2}\dfrac{1}{(r\log  b)^2}rdrd\theta=\dfrac{2\pi}{(1+\beta^2)\log b},$$
which proves (\ref{logspirals}).

 Consider the mapping $g=f(re^{i(-\beta\log r +\theta)})$ of the annulus $R_{1b}$ in the polar coordinates $Q=\{(r,\theta):  1\leq r\leq b, 0\leq\theta<2\pi \}$ on $Q^\prime. $ Then $M_2(\Gamma'_{\beta})=\int_0^{2\pi} \left( \int_1^b \dfrac{|g_r|^2}{J_g} dr  \right)^{-1}d\theta. $ Since $|g_r|^2=|f_z+f_{\bar z} e^{i\theta_0}|^2,$ $J_g=rJ_f$, and by Definition~\ref{directional}, we obtain (\ref{logspirals}).
\end{proof}

Since $M_2(\Sigma')\leq \left( M_2(\Gamma'_\beta)\right)^{-1},$ after applying the Cauchy Schwarz inequality one obtains   estimates for $M_2(\Sigma'),$ which depend on $\beta$  :
$$\label{upperestimate}
M_2(\Sigma')\leq \dfrac{1+\beta^2}{(2\pi)^2} \int_{R_{1b}} D_{f,\theta_0}\dfrac{dm}{|z|^2}. 
$$ 

Since $M_2(\Sigma')=M(R^\prime)$ the above inequality provides a  conformally invariant upper estimate for $M(R^\prime),$ which is more general than the one obtained in (\ref{ringdomainestimaes}). After simple calculations, it can be shown that it is equivalent to the one obtained in~\cite[Corollary 3.4]{Brakalova2}.


\section{Rodin's theorem in Euclidean spaces}\label{RodinS}

\subsection{Module of a system of measures in  Euclidean spaces} \label{measures}
The module of a systems of measures, introduced  in Definition \ref{FugledeM}, has the following properties, based on ~\cite[Chapter 1]{Fug}:
\begin{proposition}\label{modulusprop}
Let $(X,\mathfrak{M},m)$ be an abstract measure space where $m$ is a fixed measure defined on a subalgebra of  $\mathfrak{M}$. 
Denote by $\mathcal{M}$ the system of all measures $\mu$ in $X$, whose 
domains of definition contain $\mathfrak{M}$ and by $\bar{\mu}$ the completion of the measure $\mu$.
Then the following properties hold
\begin{enumerate}
\item $M_p(E)\leqslant M_p(E')$ if $E\subset E'$ and $E,E'\subset \mathcal M$;
\item $M_p(E)\leqslant\sum\limits_{i=1}^{\infty}M_p(E_i)$ if $E=\bigcup\limits_{i=1}^{\infty}E_i$, and $E_i\subset \mathcal M$;
\item If $A\subset X$ and $\bar{m}(A)=0$, then $\bar{\mu}(A)=0$ for $M_p$-a.a. 
$\mu\in\mathcal{M}$;
\item If $\rho\in L^p(X,\bar{m})$, then $\rho$ is $\bar{\mu}$-integrable for $M_p$-a.a. $\mu\in\mathcal{M}$;
\item If $\|\rho_i-\rho\|_{L^p(X,\bar{m})}\rightarrow 0$, then there is a subsequence $\rho_{i_j}$, such that
\[
\int_{X}|\rho_{i_j}-\rho|\,d\bar{\mu}\rightarrow0 \quad \text{for $M_p$-a.e. $\mu\in\mathcal{M}$};
\]
\item Let $E\subset\mathcal{M}$. Then $M_p(E)=0$, if and only if, there exists a non-negative function $\rho\in L^p(X,m)$,
such that
\[
\int_{X}\rho\,d\mu=+\infty \qquad \text{for every $\mu\in E$};
\]
\item If $p>1$ and $E\subset\mathcal{M}\setminus \{\mu\equiv 0\}$, then there exists a non-negative function $\rho$, such that
\[
\int_{X}\rho^p\,dm=M_p(E),\quad\text{and $\int_X\,\rho\,d\mu\geqslant 1$ for $M_p$-a.e. $\mu\in E$};
\]
\item For the measures, which are restrictions of the Hausdorff measure to compact sets, the following 
is true:
 if $p\geqslant 2$, $E_1\subset E_2\subset\cdots$ are sets of complete measures and $E=\bigcup E_i$, then 
\[
 M_p(E)=\lim\limits_{i\to\infty}M_p(E_i).
 \]
\end{enumerate}
\end{proposition}

 Badger, see \cite{Badger}, extended Beurling's criterion  (Theorem A)  to   Fuglede's $p$-module of systems of measures, as a sufficient and necessary condition. Let $X=\mathbb R^n$, let $\mathfrak M$ be a Borel $\sigma$-algebra in the topology defined by the Euclidean metric, and let $m$ be the Lebesgue measure.

\begin{theorem}[\cite{Badger}] \label{Badgertheorem} Let $E$ be a measure system in $\mathbb{R}^n$ and let $\rho$ be an admissible function for $E$ such that 
$\rho\in L^p(\mathbb{R}^n)$, $1<p<\infty$. Then $\rho_0$ is the extremal function for the $p$-module of $E$ if and only if, there exists a measure system $F$ such that
\begin{itemize}
\item $M_p(F\cup E)=M_p(E)$;
\item $\int_{\mathbb R^n} \rho_0 \,d\nu=1$ for every $\nu\in F$;
\item For all real valued   functions $g\in L^p(\mathbb{R}^n)$ the condition $\int_{\mathbb R^n}  g \,d\nu \geq 0,$ for all $\nu \in F,$ implies
$\int_{\mathbb R^n} g \rho_0^{p-1} \,dm\geq 0$.
\end{itemize}
\end{theorem}
\noindent
Analogous theorem holds for $p=1$.

\begin{remark}\label{extremaltB} 
In the same manner, as in the plane, one can see from the proof of Theorem~\ref{Badgertheorem}, that we only need to consider $g=\rho-\rho_0$ to get extremality for $\rho_0$ and  the systems of measures $E$.
\end{remark}

\subsection{Rodin's theorem for families of connecting curves in a condenser}

Here we state an extension of Rodin's result~\cite[Theorem 14]{Rodin} to the $p$-module of families of curves in $\mathbb R^n,$ connecting the two plates $D_0$ and $D_1$ of a condenser $(\Omega;D_0,D_1),$ introduced at the beginning of Section \ref{introduction}.

Recall that for any Borel set $B$, we use the Hausdorff measure defined
by $H^d(B)=\lim\limits_{\delta\to 0} H_{\delta}^{d}(B)$, where
\[
H_{\delta}^{d}(B)=\inf\{\sum_{i=1}^{\infty}(\text{diam}\, U_i)^d\colon \text{$U_i$ are open, }\bigcup_{i=1}^{\infty}U_i\supseteq B,\text{diam}\, U_i<\delta\}.
\]
We observe that the $n$-dimensional Lebesgue measure coincides with the $H^n$-Hausdorff measure in $\mathbb R^n$ for all $n$, up to a constant which we choose to~be~1.
The number $d$ is called the Hausdorff dimension of $B.$

Let $D \subset  \mathbb R^{n} $ be an $(n-1)$-Hausdorff dimensional compact and let $({x},t)\mapsto \mathbf{u}({x},t),$  where ${x}\in D$ and $ t\in [a,b]$, be an embedding that parametrizes 
a condenser $(\Omega;D_0,D_1)$  in $\mathbb R^{n},$  so that $D_0= \mathbf{u}({x},a)$ and $D_1= \mathbf{u}({x},b)$.  We assume that $\mathbf{u}=(u_1,\dots, u_n)$ is a $C^1$-smooth homeomorphism in $D\times [a,b]$ with a positive Jacobian $J_{\mathbf{u}}.$ Let $f$ be a $C^1$-smooth homeomorphism  of $\Omega$ onto  $\Omega'\subset \mathbb R^{n}$, such that the Jacobian $J_f$  is  strictly positive  in $\Omega$. 
 Denote by $I_f=J_{\mathbf{u}}J_f$.
 
 Let $\Gamma_0$
be the family of  curves $$v_x(t)=\{\mathbf{u}({x},t)\colon t\in [a,b], \,\text{${x}\in D$ is fixed}\},$$ and let  $c_x(t)=f(v_x(t))\subset \Omega'$. Then 
\[
f(\Gamma_0)=\{c_x\colon [a,b]\stackrel{\mathbf{u}}{\longrightarrow}\Omega\stackrel{f}{\longrightarrow} \Omega',\,x\in D\}. 
\]
We write $\dot{c}_x=\frac{\partial}{\partial t}c_x$.

Our generalization of Rodin's theorem is as follows.

\begin{theorem}\label{Rodin1}
Set $1/p+1/q=1$, $p,q>1$, and let
\[
\ell(x)=\int_a^b\left(\frac{|\dot{c}_x|}{I_f}\right)^{q}I_f\,dt,\quad x\in D.
\]
Then 
\[
\rho_0(y)=\frac{1}{\ell(x)}\left(\frac{|\dot{c}_x|}{I_f}\right)^{\frac{1}{p-1}}\circ f^{-1},\quad y=f\circ \mathbf u(x,t)\in \Omega',
\]
is the extremal function for the $p$-module $M_{p}(f(\Gamma_0))$ of the family $f(\Gamma_0)$, where
$$M_{p}(f(\Gamma_0))=\int_{\Omega'}\rho_0^p\,dy=\int_{D} \ell^{1-p}\,dH^{n-1}(x).$$
$f(\Gamma_0)$ is the complete extremal family of curves for $M_{p}(f(\Gamma_0)).$
\end{theorem}
 
 The proof of the above theorem is given in Section \ref{proofRodin1}. 
 
 The only result known to us, that computes explicitly the weighted $p$-module of a family of curves in $R^n$ and provides the almost extremal function for it, belongs to Ohtsuka, see Theorem 3.4.3 in \cite{Ohtsuka}. The author studies a family of curves in a tube, which are trajectories of a solenoidal vector field. It is clear that, while considering different settings, Ohtsuka's result and Theorem \ref{Rodin1}, should lead to equivalent formulas, when the settings overlap.
 
Observe that  Theorem~\ref{Rodin1} applies to the following special cases. 

\begin{itemize}
\item $\Omega=\{t\in [a,b],\, x\in D\}$ is a cylinder  in $\mathbb R^{n}$, where $D$ is a compact set in $\mathbb R^{n-1}$ and $\mathbf u\equiv$id. Then, $I_f=J_f$. 
\item A spherical ring domain $\Omega=R_{ab}$ in $\mathbb R^n$ bounded by the concentric spheres $S_a$ and $S_b$  of radii $a$ and $b$ respectively.
Then $D=S_1$ is the unit sphere in $\mathbb R^n$, $\mathbf u(x,t)=tx$, and $I_f=t^{n-1}J_f$.
\item A conical cylinder $\Omega=\{(\beta tx,t):\,\,x\in D\subset \mathbb R^{n-1},\,t\in[a,b],\,\beta>0\}$. Then $I_f=(\beta t)^{n-1}J_f$
\end{itemize}


\subsection{Modules of families of separating sets in a condenser}\label{sep_sets}

Let us define separating sets in a topological sense and describe the associated system of measures. Let  $D_0$ and $D_1$ be disjoint compact sets in the closure $\overline\Omega$ of a bounded open set $\Omega \subset \mathbb R^n$.  We denote $\Omega^*=\Omega\cup D_0\cup D_1$.

\begin{definition}\label{separating set0}
\index{separating set0}
We say that a set $\sigma$ separates $D_0$ from $D_1$ in $\Omega$ if 
\begin{itemize}
\item $\sigma\cap\Omega$ is closed in 
$\Omega$;
\item There are disjoint sets $U_1$ and $U_2$ which are open in $\Omega^*\setminus\sigma$ such that 
$\Omega^*\setminus\sigma=U_1\cup U_2$, $D_0\subset U_1$ and $D_1\subset U_2$.
\end{itemize}
\end{definition}

Let $\Sigma$ denote the class of all sets that separate $D_0$ from $D_1$. With every $\sigma\subset\Sigma$ we associate a complete measure $\mu$ in the following way: for every Hausdorff $H^{n-1}$-measurable set $A\subset \mathbb R^n$ we define 
\[
 \mu(A)={H}^{n-1}(A\cap\sigma\cap\Omega).
\]
From the properties of Hausdorff measure, it is clear that the Borel sets in $\mathbb R^n$ (here $\sigma\cap\Omega$ is closed in $\Omega$, and therefore, is Borel) are 
$\mu$-measurable, hence the module $M_q(\Sigma)$ of $\Sigma$  is the module of a family of measures $E$ in Definition~\ref{FugledeM}. We will use the two notations notations $M_q(E)$ and $M_q(\Sigma)$ interchangeably.

The next theorem states an analogue to Theorem~\ref{Rodin1}  for separating sets. 
Let $(\Omega;D_0,D_1)$ be a condenser in $\mathbb R^{n}$ defined as previously.
We define $\Sigma_0$ to be a family of  sets  $\sigma_t=\{\mathbf{u}(x,t): \,\,x\in D,\,t\in [a,b]\text{ is fixed}\}$ that separate the plates  $D_0=\mathbf{u}(x,a)$ and $D_1=\mathbf{u}(x,b)$. Assume that $f$ is $C^1$-smooth orientation preserving homeomorphism $\Omega\to \Omega'\subset \mathbb R^n$. 
 
 \begin{theorem}\label{Rodin2}
 Let $\frac{1}{p}+\frac{1}{q}=1$, and let
 \[
 \ell(t)=\int_{\sigma_t}\left(|(\nabla  (t(f^{-1}))\circ f(\mathbf{u}(x,t))|\right)^{p}I_f\,d\mu.
 \]
 Then,
 \[
 \rho_0=\frac{1}{\ell( t(f^{-1})}\left(|(\nabla ( t(f^{-1}))|\right)^{1/(q-1)}
 \]
 is the extremal metric for the $q$-module of $\Sigma'_0=f(\Sigma_0)$, and the $q$-module is equal to $M_{q}(\Sigma'_0)=\int_{\Omega'}\rho_0^q\,dy=\int_a^b \ell^{1-q}dt$.
 \end{theorem}
\noindent 

\medskip

The condition of the regularity for the mapping $f$ and the set $D$ can be relaxed according to the discussion in Section~\ref{coareasection}.

\medskip

Let us give simple but illustrative Examples \ref{ex33} and \ref{ex44} of calculation of the extremal functions and families for
the module of families of connecting curves and separating sets in the cylinder   $\Omega=D\times [a,b]$. Then, Examples \ref{e3} and \ref{e7} illustrate Theorem~\ref{Rodin1} for a cylinder and for a spherical ring domain.

\begin{example}\label{ex33} Let $\Gamma$ be a family of all locally rectifiable curves connecting the base sides $D_0$ and $D_1$ of the cylinder $\Omega$, and let $\Gamma_0\subset \Gamma$
be the family of  intervals $$v_x(t)=\{(x,t)\colon t\in [a,b], \,\text{$x\in D$ is fixed}\}.$$  
Observe that the function $\rho_0(x)=\frac{1}{b-a}$, $x\in \Omega$ is extremal for $\Gamma$ and
 $$M_p(\Gamma)=M_p(\Gamma_0)=\frac{H^{n-1}(D)}{(b-a)^{p-1}}.$$ 
\end{example}
 
\begin{example}\label{ex44} 
Let $\Sigma$ be a family of Borel  sets $\sigma$ separating the base sides $D_0=\sigma_a$ and $D_1=\sigma_b$ of the cylinder $\Omega$, and let $\Sigma_0\subset \Sigma$ 
be the family of  sets $$\sigma_t=\{(x,t)\colon x\in D, \,\text{$t\in [a,b]$ is fixed}\}.$$ 
By $E$ we denote the family $H^{n-1}$-Hausdorff
 measures in $\Omega$ associated with  $\sigma\in \Sigma$ and $E_0\subset E$ are  $H^{n-1}$-Hausdorff measures associated with $\sigma_t\in \Sigma_0$.  
 It is easy to see that $\varrho_0=\frac{1}{H^{n-1}(D)}$ is the extremal function for both $\Sigma$ and $\Sigma_0\subset \Sigma$, satisfies all other conditions of  Theorem~\ref{Badgertheorem} and
\[
 M_q(\Sigma)=\frac{b-a}{(H^{n-1}(D))^{q-1}}.
\] 
The classical equality $(M_p(\Gamma))^q(M_q(\Sigma))^p=1$ holds.
\end{example}

\begin{example}\label{e3} Now let us calculate the module in the image $\Omega'$ of $\Omega$ with $a=0$, $b=r$, under a `horizontal' shear transform $f$ given by 
$$f(x,t)=(x_1+\beta t,x_2,\dots,x_{n-1},t).$$
Let $\Gamma_0$ be as in Theorem~\ref{Rodin1} and $\Gamma_0'=f(\Gamma_0)$. Then the Jacobian  $J_f=1$  and the norm of the derivative is $|\dot{c}_x|=\sqrt{1+\beta^2}$. Besides, $\ell=(1+\beta^2)^{q/2}r$, 
$\rho_0=\frac{1}{r \sqrt{1+\beta^2}}$, and the $p$-module becomes 
\[
M_p(\Gamma'_0)=\frac{H^{n-1}(D)}{(1+\beta^2)^{p/2}r^{p-1}}.
\]

At the same time, let us calculate the $q$-module in Theorem~\ref{Rodin2} under the same transform. Obviously,  $\ell=\mu(D)$ and
\[
M_q(\Sigma_0')=\frac{r}{(H^{n-1}(D))^{q-1}}.
\]
We see that $(M_p(\Gamma'_0))^q(M_q(\Sigma'_0))^p\neq 1$, and the function $\rho_0=\frac{1}{r \sqrt{1+\beta^2}}$ is no longer admissible for a larger family $\Gamma'=f(\Gamma)$   while other conditions of Beurling-Badger's criterion applied to the families $\Gamma'$ and $\Gamma_0'$ hold. The statements on $\Sigma$, $\Sigma_0$ and $\Sigma'$, $\Sigma'_0$ remain as in Example~\ref{ex44} and in some cases can give interesting estimates as in Example~\ref{e4}. At the same time, using monotonicity of modules we can derive
estimates for the module of the whole family of curves connecting $D_0$ and $D_1$ as
\[
\frac{H^{n-1}(D)}{(1+\beta^2)^{p/2}r^{p-1}}=M_p(\Gamma'_0)\leq M_p(\Gamma')=(M_q(\Sigma'))^{p/q}\leq (M_q(\Sigma'_0))^{p/q}=\frac{H^{n-1}(D)}{r^{p-1}}.
\]
\end{example}

 \begin{example}\label{e7} Let us first give known expressions for the $p$-module of the family of locally rectifiable curves $\Gamma$ connecting $S_a$ and $S_b$ in the spherical ring domain $R_{ab}$ bounded by $S_a$ and $S_b$. The function
 \[
 \rho_0(x)=
 \begin{cases}
 \frac{|p-n|}{p-1}\left| b^{\frac{p-n}{p-1}}-a^{\frac{p-n}{p-1}}\right|^{-1}\left|\nabla |x|^{\frac{p-n}{p-1}}\right|,& \text{for $p\neq n$},\\
 (\log b/a)^{-1}|\nabla \log|x||,& \text{for $p= n$},
 \end{cases}
 \]
 is admissible for $\Gamma$ and $\int_{\gamma\in \Gamma_0}\rho_0\,ds=1$.  Besides, 
 \[
M_p(\Gamma)= \int_{R_{ab}}\rho_0^p\,dy= \begin{cases}
 \left(\frac{|p-n|}{p-1}\right)^{p-1}\left| b^{\frac{p-n}{p-1}}-a^{\frac{p-n}{p-1}}\right|^{1-p}\omega(S_1),& \text{for $p\neq n$},\\
 (\log b/a)^{1-n}\omega(S_1),& \text{for $p= n$}.
 \end{cases}
 \]
 We will prove analogous formulas in details  later for a more general case of polarizable groups, see also \cite{Geh62, Shlyk90, zim66}.
 
 Let us consider a spherical ring $R_{1r}$ bounded by the unit sphere $S_1$ and a sphere $S_r$ of radius $r,$ $1<r<\pi$. Recall that $R_{1r}$ is given by  the spherical transformation $G\colon  (\theta_1,\dots,\theta_{n-1},t)\to (x_1,\dots,x_n)$, where $\theta_1\in [0,2\pi)$, $\theta_k\in [0,\pi]$, $k=2,3,\dots, n-1$,  $t\in [1,r]$, and   
 \begin{eqnarray*}
 x_1 &=& t\sin \theta_1\sin\theta_2\dots\sin \theta_{n-1},\\
 x_2 &=& t\cos \theta_1\sin\theta_2\dots\sin \theta_{n-1},\\
 x_3 &=& t\cos \theta_2\sin\theta_3\dots\sin \theta_{n-1},\\
 \dots &\dots&\dots\\
 x_n &=& t\cos \theta_{n-1}.
 \end{eqnarray*}
 The Jacobian $J_{G}$ of the spherical transformation $G$
 is 
 \[
 J_{G}(\theta_1,\dots,\theta_{n-1},t)=t^{n-1}\omega=t^{n-1}\sin \theta_2\sin^2\theta_3\dots\sin^{n-2} \theta_{n-1}.
 \]
The line element $ds$ is computed as
 \begin{eqnarray*}
 ds^2=dt^2&+&t^2(d\theta_{n-1}^2+\sin^2\theta_{n-1}\,d\theta_{n-2}^2+\sin^2\theta_{n-1}\sin^2\theta_{n-2}\,d\theta_{n-3}^2\\
 &+&\dots +\sin^2\theta_{n-1}\sin^2\theta_{n-2}\dots \sin^2\theta_{2}\,d\theta_{1}^2).
 \end{eqnarray*}
Let us define the twisting map  $f$ of $R_{1r}$ by the shear transform
$$G^{-1}\circ f\circ G\colon (\theta_1,\theta_2,\dots,\theta_{n-1},t)\to ((\theta_1+t-1),\theta_2,\dots,\theta_{n-1},t),$$ i.e., the boundary
sphere $S_1$ remains unchanged while the  spheres $S_t$ rotate by an angle $t-1$, $t\in (1,r]$ by the element of SO$(n)$
\[
\left(
\begin{array}{rrrrrr} \cos(t-1) & \sin(t-1) & 0 & 0 & \dots & 0 \\ 
-\sin(t-1) & \cos(t-1) & 0 & 0& \dots & 0 \\ 
0 & 0 & 1 & 0 & \dots & 0\\
\dots & \dots &  \dots  & \ddots & \dots & \dots  \\  
0 & 0 & 0 & 0 & \dots & 1
\end{array}
\right).
\]
Observe that $J_f=J_{G^{-1}\circ f\circ G}=1$, and the radial intervals from $\Gamma_0$ are mapped onto the curves $c(t)$ with $|\dot{c}|=\sqrt{1+t^2}$. Theorem~\ref{Rodin1} implies that
\[
\ell(x)=\int_1^r\left(1+t^2\right)^{q/2}I_f^{1-q}\,dt=\int_1^r\left(1+t^2\right)^{q/2}t^{(n-1)(1-q)}\,dt=:K,
\]
and the $p$-module of $f(\Gamma_0)$ is
\[
M_p(f(\Gamma_0))=\int_{S_1} \ell^{1-p}dx=K^{1-p}\int_{S_1} dx=K^{1-p}\omega(S_1).
\]
Having in mind the previous example and the monotonicity of the module, we come to correct
inequalities for the following integrals
\[
\int_1^r\left(1+t^2\right)^{\frac{p}{2(p-1)}}t^{\frac{n-1}{1-p}}\,dt\geq  \frac{p-1}{|p-n|}\left( r^{\frac{p-n}{p-1}}-1\right),\quad p> 1, \quad p\neq n,
\]
and
\[
\int_1^r\frac{\left(1+t^2\right)^{\frac{n}{2(n-1)}}}{t}\,dt\geq  \log r,\quad p= n\geq 2.
\]
 \end{example}
Analogously we can define an automorphism $f$ of $R_{1r}$ by the twisting map
$$G^{-1}\circ f\circ G\colon (\theta_1,\theta_2,\dots,\theta_{n-1},t)\to ((\theta_1+\beta\log t),\theta_2,\dots,\theta_{n-1},t),$$ i.e., the boundary
sphere $S_1$ remains unchanged while the  spheres $S_t$ rotate by an angle $\beta\log t$, $t\in (1,r]$. Then the radial intervals from $\Gamma_0$ are mapped onto the curves $c(t)$ with $|\dot{c}|=\sqrt{1+\beta^2}=$const, and
\[
\ell(x)\equiv
\begin{cases}
(1+\beta^2)^{\frac{q}{2}}\,\,{\displaystyle \frac{r^{\frac{p-n}{p-1}}-1}{p-n}(p-1)}, & \text{for $p\neq n$},\\
(1+\beta^2)^{\frac{n}{2(n-1)}}\log r, & \text{if $q=\frac{n}{n-1}$ and $p=n$}.
\end{cases}
\]
Hence,
\[
M_p(f(\Gamma_0))=
\begin{cases}
\frac{1}{(1+\beta^2)^{p/2}} \left(\frac{|p-n|}{p-1}\right)^{p-1}\left| r^{\frac{p-n}{p-1}}-1\right|^{1-p}\omega(S_1),& \text{for $p\neq n$},\\
\frac{\omega(S_1)}{(1+\beta^2)^{n/2}(\log r)^{n-1}}, & \text{if $q=\frac{n}{n-1}$ and $p=n$}.
\end{cases}
\]
In particular, if $p=2$ and $n=2$, then
$
M_2(f(\Gamma_0))=\frac{2\pi}{(1+\beta^2)\log r}.
$ 
This, naturally, agrees with (\ref{logspirals}).

 \subsection{Proof of Theorem \ref{Rodin1} and \ref{Rodin2}}
 \subsubsection{Proof of Theorem \ref{Rodin1}}\label{proofRodin1}
 \begin{proof}
Let us first observe that $\int_{c_x}\rho_0\,ds=1$. Indeed,
\begin{eqnarray*}
\int_{c_x}\rho_0\,ds&=&\int_a^b(\rho_0\circ f\circ \mathbf u)|\dot{c}_x|\,dt=\frac{1}{\ell}\int_a^b\left(\frac{|\dot{c}_x|}{I_f}\right)^{\frac{1}{p-1}}|\dot{c}_x|\,dt\\
&=&\frac{1}{\ell}\int_a^b\left(\frac{|\dot{c}_x|}{I_f}\right)^{q}I_f\,dt=1,\quad \text{for all}\quad x\in D.
\end{eqnarray*}
Therefore, $\rho_0$ is admissible for $f(\Gamma_0)$ and 
\begin{equation}\label{modd1}
M_p(f(\Gamma_0))\leq \int_{\Omega'}\rho_0^p\,dy.
\end{equation}
On the other hand,
for any $\rho$ admissible for $f(\Gamma_0),$ we have $\int_{c_x}\rho\, ds \geq 1$, and therefore
\[
\int_{c_x}(\rho-\rho_0)\, ds\geq 0.
\]
This implies that
\[
\frac{1}{\ell^{p-1}(x)}\int_a^b [(\rho-\rho_0)\circ f\circ \mathbf u\,]\,|\dot{c}_x|\,dt\geq 0.
\]
Then
\[
\int_{D} \int_a^b \left((\rho-\rho_0)\rho^{p-1}_0\circ f\circ \mathbf u\right)\,I_f\, dt\,dH^{n-1}_x\geq 0.
\]
Equivalently,
\[
\int_{\Omega'}\rho\rho_0^{p-1}\,dy\geq \int_{\Omega'}\rho_0^{p}\,dy.
\]
Thus we can apply the Remark~\ref{extremaltB},  to conclude that $\rho_0$ is extremal for the module problem in consideration and that
$
M_p(f(\Gamma_0))=\int_{\Omega'}\rho_0^{p}\,dy.
$
Now we can calculate the $p$-module as
\begin{eqnarray*}
M_p(f(\Gamma_0))&=&\int_{\Omega'}\rho_0^{p}\,dy=\int_{D} \int_a^b [\rho_0^{p}\circ f\circ \mathbf u]\,I_fdt\,dH^{n-1}_x\\
&=&\int_{D} \int_a^b\frac{1}{\ell^p} \left(\frac{|\dot{c}_x|}{I_f}\right)^{\frac{p}{p-1}}\,I_f\,dt\,dH^{n-1}_x=\int_{D} \ell^{1-p}dH^{n-1}_x.
\end{eqnarray*}
\end{proof}

The condition of regularity of the mapping $f$ required for Theorem~\ref{Rodin1} can be relaxed, so that $f$ belongs to some more general known classes of maps, see Section~\ref{nondifferentiable}. 
For example, we can assume that $f$ is homeomorphic $W^{1,1}$ Sobolev and orientation preserving in $\mathbb R^3$,  see \cite{HMaly2010}, or quasiconformal  in $\mathbb R^n$, see \cite{va71}. 
 
 
 \subsubsection{Coarea formula and change of variables}\label{coareasection}

In order to prove Theorem~\ref{Rodin2} we need a variant of the coarea and change-of-variable formulas which we formulate first for smooth maps.

Let $u\colon \mathbb R^n\to \mathbb R$ be $C^1$-smooth, $f\colon\mathbb R^n\rightarrow \mathbb R^n$ a $C^1$ orientation preserving homeomorphism. Let $D$ be $H^{n-1}$-measurable set, $\Omega=D\times (-\infty,\infty),$ and $\Omega^\prime=f(\Omega). $ Let $\sigma_r=\{{x\in\Omega: (u\circ f )(x)=r\}}$ and $\sigma'_r=\{y\in \Omega^\prime \colon u(y)=r\}$ be the level sets of $u\circ f$ and $u$ for a fixed $r,$ respectively. A variant of Sard's theorem, see~\cite{Dub57}, states that  $H^{n-1}_y(\sigma'_r\cap Z_u)=0$ for almost all $r\in\mathbb R$, where $Z_u$
is the  set of the points where the Jacobian $J_u$ vanishes.

In particular, let us consider the condenser $(\Omega;D_0,D_1)$  in $\mathbb R^{n}$  parametrised by an embedding $({x},t)\mapsto \mathbf{u}({x},t)$, where ${x}\in D\subset  \mathbb R^{n}$, $D$ is $(n-1)$-Hausdorff dimensional compact , $\mathbf{u}=(u_1,\dots, u_n)$ is $C^1$-smooth homeomorphism with positive Jacobian $J_{\mathbf{u}}$,
 and $t\in [a,b]$. The function $u$ will be defined to be $u(y)=t\circ f^{-1}(y)$ where $t(\cdot)$ is the last coordinate $t$ of $(x,t)$ of the inverse map $\mathbf{u}^{-1}$. Define $I_f=J_fJ_{\mathbf{u}}$ .

Let us recall the change of variables formula. For any Borel set $U\subset \mathbb R^n$ and for any homeomorphism $f\colon \mathbb R^n\to\mathbb R^n$ we define the pull-back of the Hausdorff measure $\mu$ by $f$ as $f^*\mu(U)\equiv(\mu\circ f)(U)=\mu(f(U))$. Then $\mu\circ f$ is absolutely continuous with respect to $\mu$ and the Radon-Nikodym theorem implies the change of variables formula 
 \[
\int_{f(U)}\rho'd\mu= \int_{U}(\rho'\circ f)\frac{d(\mu\circ f)}{d\mu}d\mu.
 \]
  
\begin{lemma}\label{ChVar} 
For a positive measurable function $\rho$ in $\mathbb R^n$ the following change of variables formula holds
\begin{equation}\label{changeofv0}
\int_{\sigma'_r}\rho dH^{n-1}(y)=\int_{\sigma_r}\frac{\rho |\nabla u|\circ f}{|\nabla(u\circ f)|}J_fdH^{n-1}(x).
\end{equation}
In addition, if $u=t\circ f^{-1}(y),$ then
\begin{equation}\label{changeofv}
\int_{\sigma'_r}\rho dH^{n-1}(y)=\int_{\sigma_r} (\rho\,|(\nabla (t(f^{-1}))\circ f(\mathbf{u}(x,t))|) I_fdH^{n-1}(x).
\end{equation}
\end{lemma}
\begin{proof}  Indeed, the conditions on the function $u$ and the mapping $f$  ensure that both $u$ and $u\circ f$ are $C^1$-smooth.   For a fixed $r$ denote by $B_r=\{x \in  \Omega \colon u\circ f(x)\leq r\}$ and $B'_r=\{y\in \Omega' \colon u(y)\leq r\}$. 

By the co-area formula (see, e.g., \cite{Maly2003}) we have that
\[
\int_{B'_r} \rho |\nabla u| dy= \int_{-\infty}^r  \int_{\sigma'_r}\rho dH^{n-1}(y) dr. 
\]
Then
\begin{equation}\label{coarea}
\frac{\partial}{\partial r}\int_{B'_r}\rho |\nabla u| dy=\int_{\sigma'_r}\rho dH^{n-1} (y)
\end{equation}
holds as the above derivative exists for almost all $r$, see, e.g.,~\cite[pp. 29--33]{Morgan}.

By the change of variables formula (see \cite{Hajlasz0}) $$ \int_{B'_r}\rho |\nabla u| dy=\int_{B_r}(\rho|\nabla u|)\circ f J_f dx .$$  Applying  the co-area formula, once more and in a similar fashion as before, to the right hand side of the above equation and differentiating with respect to $r$ gives

\begin{equation}\label{coarea1}
\frac{\partial}{\partial r}\int_{B_r}(\rho|\nabla u|)\circ f J_f  dx=\int_{\sigma_r}\dfrac{(\rho |\nabla u|)\circ f }{|\nabla (u\circ f)|} J_fdH^{n-1}(x) .
\end{equation}

Comparison of the right-hand sides in \eqref{coarea} and \eqref{coarea1} leads to the first statement of  Lemma \ref{ChVar}.

To prove the second part of  Lemma \ref{ChVar}, observe that we set $u=t(f^{-1}).$  Thus $\sigma_r=\{x\colon (x_1,\dots,x_{n-1})\in D,\,t=r\in(a,b)\}$ and
$\sigma_r'$ is a level set of the function $t(f^{-1})$, i.e.,  $\sigma'_r=\{y\in f(\Omega)\colon t(f^{-1})=r\}.$ Then we must change $f\to f\circ \mathbf{u}$ and $u\to t(f^{-1})$ in (\ref{changeofv0}). 
The gradient in the denominator $|\nabla(u\circ f)|$ in (\ref{changeofv0})  becomes $|\nabla(t(f^{-1})\circ f\circ \mathbf{u})|=|(0,\dots,1)|=1$ from where follows (\ref{changeofv}).
\end{proof}

 Lemma~\ref{ChVar} implies that in our case the Radon-Nikodym derivative becomes
 \begin{equation}\label{sarea}
 \frac{d(\mu\circ f)}{d\mu}=|(\nabla (t(f^{-1}))\circ f(\mathbf{u}(x,t))|I_f(x).
 \end{equation}

\subsubsection{Proof of Theorem \ref{Rodin2}}


 \begin{proof}
The proof is similar to that of Theorem~\ref{Rodin1} where we substitute Radon-Nikodym derivative \eqref{sarea} in place of $|\dot{c}_x|$ in the claim on admissibility of $\rho_0$, and use  Lemma~\ref{ChVar} when the change-of-variable formula (\ref{changeofv}) is applied.
First we show that $\rho_0$ is an admissible function:
$$
\int_{\sigma'_t}\rho_0dH^{n-1}=\int_{\sigma'_t} \frac{1}{\ell( t(f^{-1}))}\left(|\nabla  t(f^{-1})|\right)^{1/(q-1)}dH^{n-1}=
$$
$$
\int_{\sigma_t} \frac{1}{\ell( t(f^{-1}))}\left(|\nabla  t(f^{-1})|\right)^{1/(q-1)}|\nabla  t(f^{-1})|I_fdH^{n-1}=
$$
$$
\dfrac{1}{\ell( t(f^{-1}))}\int_{\sigma_t} |\nabla  t(f^{-1})|^pI_fdH^{n-1}=1.
$$

Now let $\rho$ be any admissible function for $E'_0,$ i.e. $\int_{\sigma'_t}\rho dH^{n-1}\geq 1.$ Since $\int_{\sigma'_t}\rho_0dH^{n-1}\geq 1$, we have
$$
\dfrac1{\ell^{q-1}}\int_{\sigma'_t}(\rho-\rho_0)dH^{n-1}\geq 0 .
$$
and by Lemma \ref{ChVar}
$$
\dfrac1{\ell^{q-1}}\int_{\sigma_t}((\rho-\rho_0) |\nabla  t(f^{-1})|)\circ f  I_f dH^{n-1}\geq 0 .
$$
Since $|\nabla  t(f^{-1})|= \ell^ {q-1} \rho^{q-1}\circ f\circ\mathbf{u},$
it follows that
$$
\int_a^b \int_{\sigma_t}((\rho-\rho_0) \rho_0^{q-1})\circ f \circ\mathbf{u}\,I_f dH^{n-1} dt\geq 0.
$$
By Fubini's theorem and change of variables formula we obtain that
$
 \int_{\Omega'}(\rho-\rho_0) \rho_0^{q-1}dy \geq 0,
$
and therefore applying Remark~\ref{extremaltB}, we conclude that $\rho_0$ is extremal for the module problem under consideration and that $M_q(\Sigma'_0)=\int_{\Omega'} \rho_0^{q}dy.$
\end{proof}


\subsubsection{Theorem \ref{Rodin1} and \ref{Rodin2} for non-smooth homeomorphisms}\label{nondifferentiable}

The conditions we imposed on the function $u$ and the map $f$ can be relaxed. Let us analyse the following ingredients of the above lemmas. We need the following properties:
\begin{itemize}
\item[(i)] The function $u\colon \mathbb R^n\to \mathbb R$ and the homeomorphism $f\colon \mathbb R^n\to \mathbb R^n$ must be such that $u$ and 
the superposition $u\circ f$ are of the same type of regularity;
\item[(ii)]  The homeomorphism $f\colon \mathbb R^n\to \mathbb R^n$ and its inverse $f^{-1}$ must be of the same type of regularity;
\item[(iii)]  The critical set for the level sets $\sigma_r$ and $\sigma'_r$ must have the $H^{n-1}$-measure zero. 
\end{itemize}

First, let us  describe the level sets of the function $u$.

\begin{definition}
Given integers $1\leq k\leq n$, we say that a Borel set $B\subset \mathbb R^n$ is {\it countably Hausdorff $H^k$-rectifiable} if there exists a sequence of Lipschitz maps
$\psi_i\colon E_i\subset \mathbb R^k\to \mathbb R^n$ such that  $H^k\left(B\setminus \bigcup_{i=1}^{\infty}\psi_i(E_i)\right)=0$. 
\end{definition}

The set $B\subset \mathbb R^n$ is countably $H^k$-rectifiable if and only if there exists a sequence of $k$-dimensional $C^1$-smooth manifolds $M_1, M_2,\dots$, such
that \linebreak $H^k\left(B\setminus \bigcup_{i=1}^{\infty}M_i\right)=0$, see~\cite{Fed,Hajlasz}. A version of Sard's theorem tells us that if $u\colon \mathbb R^n\to\mathbb R$ is Lipschitz, then
for almost every point $r\in \mathbb R$, $u^{-1}(r)$ is  countably $H^{n-1}$-rectifiable, see \cite{Hajlasz}. 
If $u\in W^{1,p}(\mathbb R^n,\mathbb R)$, then there exists a Borel representative of $u$ such that $u^{-1}(r)$ is countably $H^{n - 1}$-rectifiable for almost all $r$, see \cite{AS, Boya05, Fed, Hajlasz, Maly2003} for more on rectifiable sets. We remark that this analogue of Sard's theorem is not enough to state the coarea formula. The following definition is found in~\cite{Resh69}.

\begin{definition}
A function $u\colon \mathbb R^n\to\mathbb R$ is called {\it $p$-quasicontinuous} if it is Borel measurable and continuous on a set $\mathbb R^n\setminus U$ where $U$ is a set of arbitrary small $p$-capacity.
\end{definition}

Reshetnyak's theorem~\cite{Resh69} claims that
any $W^{1,p}(\mathbb R^n, \mathbb R)$ Sobolev function, $p\geq 1$, has a $p$-quasicontinuous representative for $1\leq p \leq n$, or a continuous representative for $p> n$.
Now  let $u\colon \mathbb R^n\to \mathbb R$ be a $p$-quasicontinuous representative of a $W^{1,p}$-Sobolev function, $p\geq 1$, and let
 $g\colon \mathbb R^n\to \mathbb R$ be positive and measurable.
Then the  coarea formula 
\eqref{coarea}
holds, where  $\nabla$ means the distributional gradient.

We discuss now the composition operators. We say that a locally Lipschitz function $u\colon \mathbb R^n\to \mathbb R$ satisfies the polynomial growth
condition for its partial derivatives almost everywhere in $\mathbb R^n$, if
\[
\left|\frac{\partial u}{\partial x_j}(x)\right|\leq a(1+|x_j|^{\nu}),\quad \text{a.e. in $\mathbb R^n$, $j=1,\dots, n$},
\]
where 
$
\nu=\frac{n(p-r)}{r(n-p)}, \quad 1\leq r\leq p<n<\infty,
$
 and $a>0$ is a constant. Given a function $u\colon\mathbb R^n\to\mathbb R$, and a map $f\in W^{1,p}(\Omega, \mathbb R^n)$, $p\geq 1$ let us define a composition operator  $T_u$ by $T_u(f)=u\circ f$. The operator $T_u$ maps   $W^{1,p}(\Omega, \mathbb R^n)$ into $W^{1,r}(\Omega, \mathbb R)$, if and only if,  $1\leq p<n $, the domain $\Omega$ is bounded domain in $\mathbb R^n$ satisfying the cone condition, and $u$ is locally Lipschitz having its partial derivatives  of polynomial growth, see Marcus and Mizel~\cite{Marcus79}. For $p>n$,
the polynomial growth of the partial derivatives can be omitted.

A different composition operator  $T_f\colon W^{1,p}(\Omega',\mathbb R)\to W^{1,p}(\Omega,\mathbb R)$ defined
by  $T_f(u)=u\circ f$ was considered in a series of papers \cite{Gold84, Vodopis76, Vodopis89}. It was proved that the necessary and sufficient condition
for $T_f$ to induce an isomorphism of $W^{1,p}$ spaces is that $f$ is quasiisometric.
Ukhlov and Vodopyanov~\cite{Vodopis2002, Vodopis2005}  studied also the operator  $T_f\colon L^{1,p}(\Omega',\mathbb R)\to L^{1,q}(\Omega,\mathbb R)$, for Sobolev spaces whose Sobolev norm does not contain $\|u\|_p$. A necessary
and sufficient condition  
 involved ACL function of bounded distortion, which in particular, included the case $p=n$. Let us remark that $ACL^p=W^{1,p}$ in the sense that any $ACL^p$ function represents a function from $W^{1,p}$ and any $W^{1,p}$ function contains an $ACL^p$ representative.

Finally, let us turn to the problem of regularity of the inverse map. For this we need the definition of a function of finite distortion, see~\cite{IvaMar}.

\begin{definition}
A map $f\colon \Omega\to \mathbb R^n$ on an open set $\Omega\subset \mathbb R^n$ has finite distortion if $f\in W^{1,1}_{\text{loc}}(\Omega,\mathbb R^n)$, the Jacobian $J_f\geq 0$ almost everywhere in $\Omega$, $J_f\in L^1_{\text{loc}}(\Omega)$ and there exists some function $K\colon \Omega\to [1,\infty]$ finite almost everywhere in $\Omega$ such that
\[
\|Df(x)\|^n\leq K(x) J_f(x), \quad \text{a.e. in $\Omega$}.
\]
\end{definition}

Let $\Omega$ be an open set in $\mathbb R^n$, and let $f\in W^{1,n-1}(\Omega,\mathbb R^n)$ be a homeomorphism of finite distortion, $\Omega'=f(\Omega)$. A result by Cs\"ornyei, Hencl, and Mal\'y \cite{CMaly2010} implies that $f^{-1}\in W^{1,n-1}(\Omega',\mathbb R^n)$ and $(f^{-1})_n\in W^{1,n-1}(\Omega',\mathbb R)$.
For the results on the functions of the Sobolev class $L^{1,p}$, see~\cite{Vodopis2012}.

\medskip

\noindent
{\bf Conclusion.}
 {\it  In view of the above, the conditions on the set $\Omega\subset \mathbb R^n$, on the 
function $u$, and on the mapping $f$ in Lemma~\ref{ChVar}
can be chosen, for example, as follows.
\begin{itemize}
\item The function $u$ can be chosen to be locally Lipschitz function from $W^{1,p}$ satisfying the polynomial growth condition
for its partial derivatives for $1\leq p<n$, and the homeomorphism $f$ to be a $p$-quasicontinuous representative of a $W^{1,p}$-Sobolev map in Lemma~\ref{ChVar};
\item The homeomorphism $f$ in 
Lemma~\ref{ChVar} can be chosen to be $W^{1,n-1}$-Sobolev of finite distortion.
\end{itemize}
}



\section{Extremal measures  on polarizable groups}\label{groups}


In this section we want to prove an analogue of  Rodin's theorem and to discuss extremal functions and extremal families for the module of curves and separating sets in a geometry different from the Euclidean one, namely for a special type of Carnot groups. We start with some necessary definitions.


\subsection{Definition of polarizable groups}


\begin{definition}
\index{Carnot group}
The Carnot group $G$ is a connected, simply connected Lie group, whose Lie algebra $\mathfrak{g}$ is nilpotent and possesses a
stratification
$
\mathfrak{g}=\bigoplus_{j=1}^{l}V_j,
$
where $[V_1,V_j]=V_{j+1}$ for all $j\in\mathbb{N}$ with $V_j=\{0\}$, whenever $j>l$. The positive integer $l$
is called a step of the group. 
\end{definition}

We assume that the underlying layer $V_1$ is  endowed with an inner product $\langle.\,,.\rangle_0$, and let $X_1,\cdots,X_k$ be an orthonormal basis of $V_1$ with respect to this inner product. The vector fields $X_1,\cdots,X_k$ are usually called horizontal, and a sub-bundle $HG$ of the tangent bundle $TG$ of the group $G$ with the typical fiber $H_g G=\spn\{X_1(g),\cdots,X_k(g)\}\subset T_gG$, $g\in G$, is called a horizontal sub-bundle. As a consequence, any vector $v\in H_g G$ is also called horizontal. The inner product $\langle.\,,.\rangle_0$ on $V_1$ defines a left-invariant 
sub-Riemanian metric on $G$ through the left translations, which we denote by the same symbol. We write $\|v\|_0^2=\langle v,v\rangle_0$ for $v\in H_g G$.
Let us use the normal coordinates of the first kind, where an element $g\in G$ is identified with 
$(x_1,\cdots,x_k,t_{k+1},\cdots,t_m)\in\R^m$ by the formula
\[
g=\exp\left(\sum\limits_{i=1}^k x_iX_i+\sum\limits_{i=k+1}^m t_iT_i\right),
\]
where $T_{k+1},\cdots,T_m$ denotes a set of vectors, extending the horizontal basis $X_1,\cdots,X_k$ to the entire basis of $\mathfrak{g}$.
The stratified structure of the Lie algebra naturally defines the dilation $\delta_s$, $s>0$, that in the introduced coordinates can be written as
\begin{eqnarray*}
\delta_sg & = &\delta_s(x_1,\cdots,x_k,t_{k+1},\cdots,t_{k+\dim (V_2)},\ldots,t_{k+\sum\limits_{j=1}^{l-1}\dim(V_{j})+1},\cdots,t_{m})
\\
& = & (sx_1,\cdots,sx_k,s^2t_{k+1},\cdots,s^2t_{k+\dim (V_2)},\ldots,s^lt_{k+\sum\limits_{j=1}^{l-1}\dim(V_{j})+1},\cdots,s^lt_{m}).
\end{eqnarray*}
As a simply connected nilpotent group, admitting dilations $\delta_s$, $G$ is globally diffeomorphic to 
$\mathfrak{g}\cong\R^m$, $m=\sum_{i=1}^l\dim{V_i}$, via the exponential map, see~\cite[Proposition 1.2]{folland1982}. The number $m$ is called a topological dimension of the group. The sub-Riemannian metric induces the distance function $d_{cc}$ on $G$ in the same way as it does for a Riemannian metric. The distance function $d_{cc}$ is usually called the Carnot-Carath\'eodory distance, and the Hausdorff dimension of the metric space $(G,d_{cc})$ is equal to $Q=\sum_{i=1}^li\dim{V_i}$, see~\cite{Mitch}. The number $Q$ is also called a {\it homogeneous dimension} of $G$, and it will play an important role in the forthcoming calculations. The Haar measure on $G$ is induced by the exponential map from the Lebesgue measure on $\mathfrak{g}\cong\R^m$. 
A norm $N_G$ on the group $G$ is called homogeneous if it is a homogeneous of order one function with respect to the dilation $\delta_s$: $N_G(\delta_sg)=sN_G(g)$ for all $g\in G$.

The horizontal gradient $\nabla_0$ is a unique horizontal vector such that  
\[
\langle\nabla_0f,v\rangle_0=v(f),\quad\text{for any}\quad v\in H_gG,\quad f\in C^{\infty}(G).
\]
The horizontal gradient is expressed in the orthonormal basis $X_1,\ldots, X_k$ as  $\nabla_0f=(X_1f,\ldots,X_kf)$.

Given a domain $U\subset G$, a function $u\in C^2(U)$ is called $p$-harmonic \index{$p$-harmonic function}
if it satisfies the
$p$-sub-Laplace
equation in $U$:
\begin{equation}\label{eq:p_laplace}
 \Delta_{0,p}u:=\sum\limits_{i=1}^kX_i\left(\|\nabla_0u\|_0^{p-2}X_iu\right)=0,
\end{equation}
and
$\infty$-harmonic \index{$\infty$-harmonic function}
if it satisfies the
$\infty$-sub-Laplace
equation in $U$:
\[
 \Delta_{0,\infty}u:=\frac{1}{2}\langle\nabla_0\|\nabla_0u\|_0^{2},\nabla_0u\rangle_0=0.
\]
The derivatives of $u$ can also be understood in the generalised sense. By a result of Folland~\cite[Theorem 2.1]{folland1975},  there exists a unique fundamental solution $u_2$ in any Carnot group $G$ to
the Kohn sub-Laplacian $\Delta_{0,2}$, which is smooth away from zero and homogeneous of degree 
$2-Q:\,u_2\circ \delta_s=s^{2-Q}u_2$.
\begin{definition}\cite{balogh2001}
\label{polarizable group}
We say that a Carnot group $G$ is polarizable if the fundamental solution $u_2$ of the Kohn sub-Laplacian $\Delta_{0,2}$ has 
the property that the homogeneous norm $N_G=u_2^{\frac{1}{2-Q}}$ associated with $u_2$ is $\infty$-harmonic 
away from zero in $G$. 
\end{definition}
Examples of polarizable groups are $\R^m$, $n$-th Heisenberg group $\heis^n$, and $H$(eisenberg)-type groups introduced by Kaplan~\cite{Kaplan}, which definition will be given in Section~\ref{sec:Htype}. The main result of~ 
\cite{balogh2001} is that in any polarazible group it is possible to carry out the construction of some sort of spherical coordinates in
the same way as it has been done in~\cite{KorReim87} for the Heisenberg group. 

Let $G$ be a polarizable Carnot group, and let $N_G$ be a norm from Definition~\ref{polarizable group}. We denote by $\mathcal Z$ the characteristic set of the function $N_G$:
\[
\mathcal Z:=\{0\}\cup\{g\in G\setminus \{0\}\mid\ \nabla_0N_G(g)=0\}.
\]
The radial flow is the solution to the Cauchy initial-value problem in $G\setminus\mathcal Z$
\begin{equation}\label{radialflow}
\left\{ \begin{array}{lllllll}
 &\frac{\partial}{\partial{s}}\,\phi(s,g)=\frac{N_G(\phi(s,g))}{s}\cdot\frac{\nabla_0N_G(\phi(s,g))}{\|\nabla_0N_G(\phi(s,g))\|_0^2},\\
 &\phi(1,g)=g.
       \end{array}
\right. 
\end{equation}
\begin{proposition} \label{lemmahorizflow}\cite{balogh2001}
\renewcommand{\theenumi}{\roman{enumi}}
\noindent The flow $\phi$ satisfies the following properties:
\begin{enumerate}
 \item $N_G(\phi(s,g))=sN_G(g)$ for $s>0$, and $g\in G\setminus\mathcal{Z}$;
 \item $\|(\partial{\phi}/\partial{s})\|_0$ is independent of $s$, i.e., 
\[
\|(\partial{\phi}/\partial{s})\|_0=\frac{N_G(g)}{\|\nabla_0N_G(\phi(s,g))\|_0}=:\lambda(g)^{-1},
\]
for a non-zero real-valued function $\lambda$ on $G\setminus\mathcal{Z}$;
 \item $\det D_g\phi(s,g)=s^Q$ for $s>0$ and $g\in G\setminus\mathcal{Z}$, where $D_g\phi$ denotes the differential of the map
$\phi(s,\cdot):G\setminus\mathcal{Z}\to G\setminus\mathcal{Z}$, where $G\setminus\mathcal{Z}$ is considered as a domain in $\mathbb{R}^m$.
\end{enumerate}
\end{proposition}

\begin{definition}
An absolutely continuous curve $c\colon I\to G$ is called horizontal if the vector $\frac{d}{ds}c(s)$ is horizontal, i.e., $\frac{d}{ds}c(s)\in H_{c(s)}G$ for all $s\in I$, when it is defined.
\end{definition}
The solution $\phi$ of the Cauchy problem~\eqref{radialflow} is a horizontal curve simply because its tangent vector $\frac{\partial\phi(s,g)}{\partial s}$ is proportional to the horizontal gradient of some function, namely, of the homogeneous norm.

At the end of this section we notice that if $c\colon I\to G$ is a horizontal curve whose locus belongs to the level set of the function $N_G$, then
\[
\frac{d}{d\,\tau}N_G(c(\tau))=\langle \nabla_0 N_G,\frac{d\, c(\tau)}{d\,\tau}\rangle_0=0.
\]
We say that the horizontal gradient is orthogonal to the level set meaning that it is orthogonal to any horizontal curve (whose locus belongs to the level set) with respect to the inner product $\langle .\,,.\rangle_0$. As a consequence, we conclude that the flow $\phi$ solving Cauchy problem~\eqref{radialflow} is orthogonal to the level set of the function $N_G$, where the orthogonality is understood with respect to the inner product  $\langle .\,,.\rangle_0$ of the tangent vector to the flow and the tangent vectors to the horizontal curves lying on the level set.

Note that not all Carnot groups are polarizable. It is known that the fundamental solution to $2$-sub-Laplacian always 
exists, but it is not necessarily $\infty$-harmonic. An example of such kind of groups can be anisotropic $H$-type  groups, see reasoning in~\cite{balogh2001}. Anisotropic $H$-type groups were studied for instance in~\cite{CChGr,ChMar1}. Now we continue with examples of polarizable groups.


\subsubsection{Euclidean space}


First of all, notice that $G=\mathbb R^k$ is a Carnot group of step 1, where the Lie algebra is the space $V_1=\spn\{X_1,\ldots,X_k\}$ with $X_j=\frac{\partial}{\partial x_j}$. Since the commutation relations of $X_j$, $j=1,\ldots,k$ all vanish, the spaces $V_2=\ldots=V_l=\{0\}$. The exponential map is a map identifying $\mathbb R^k$ with $V^1$. The Kohn sub-Laplacian $\Delta_{0,2}$ is the usual Laplacian $\Delta_{2}$, whose fundamental solution $u_2(x)=|x|^{2-n}$ defines the homogeneous norm $N_{\mathbb R^k}(x)=|x|$, which is the Euclidean norm of the element $x\in\mathbb R^k$. It is trivial to check that $|\cdot|$ is an $\infty$-harmonic function. The radial flow $\phi(s,x)=sx$, $x\in\mathbb R^k$ is the solution to the corresponding Cauchy problem~~\eqref{radialflow}.


\subsubsection{H-type groups}\label{sec:Htype}


\begin{definition}
\index{$H$-type group}
We say that a Carnot group $G$ is of Heisenberg type $(H$-type$)$ if its Lie algebra $\mathfrak{g}=V_1\oplus V_2$ of $G$ 
is 2 step, and if it is endowed with an inner product 
$\langle.\,,.\rangle$, and admits a linear map $J\colon V_2\to \End(V_1)$,  compatible with the 
inner product in the following sense:
\begin{enumerate}
 \item $\langle J_ZU,V\rangle=\langle Z,[U,V]\rangle\quad\text{for all}\quad Z\in V_2\quad U,V\in V_1$ and
 \item $J_Z^2=-\|Z\|\Id\qquad\text{for all}\quad Z\in V_2,\quad\text{where}\quad\|Z\|^2=\langle Z,Z\rangle$.
\end{enumerate}
\end{definition}
The inner product $\langle.\,,.\rangle_0$ on $V_1$ is the restriction of $\langle.\,,.\rangle$ on $V_1$. 

Let $G$ be a group of $H$-type. Since the exponential map of $G$ is an analytic diffeomorphism, we can define real analytic mappings
$u\colon G\to V_1$ and $z\colon G\to V_2$ by
$g=\exp\left(u(g)+z(g)\right)$, $g\in G$.
The function
\begin{equation}\label{eq:normHtype}
 N_G(g)=\left(\|u(g)\|^4+16\|z(g)\|^2\right)^{1/4}
\end{equation}
is  a homogeneous norm on $G$. It is well known that $N_G$ is smooth on $G\backslash\{0\}$, see \cite[Theorem 2]{Kaplan}. It was shown in~\cite[Proposition 5.6]{balogh2001}  that $H$-type groups are polarizable with the norm $ N_G$, defined in~\eqref{eq:normHtype}.


\subsubsection{Heisenberg group}

An example of $H$-type groups is the Heisenberg group.
\begin{definition}\label{Heisenberg group}
The $n$-dimensional Heisenberg group $\heis^n$ is an analytic, nilpotent Lie group whose underlying manifold is 
$\R^{2n+1}$, and whose Lie algebra $\mathfrak{h}$ is graded
\begin{itemize}
 \item[(1)] $\mathfrak{h}=V_1\oplus V_2$, where $\dim(V_1)=2n$ and $\dim(V_2)=1$, and
 \item[(2)] $\mathfrak{h}$  admits the following commutation relations: $$[V_1,V_1]=V_2,\quad [V_1,V_2]=[V_2,V_2]=\{0\}.$$
\end{itemize}
\end{definition}
Choose any inner product $\langle.\,,.\rangle$ on $\mathfrak{h}$, and define the map $J\colon V_2\to\End(V_1)$ as follows. Given any $U\in V_1$, define $\ad_U\colon V_1\to V_2 $ by $\ad_UV:=[U,V]$. Then the map $J$ is the formal adjoint map $J=\ad_U^*$ given by $\langle J_ZU,V\rangle=\langle Z,\ad_UV\rangle$. 

Using normal coordinates of the first kind $g=(x,y,t)$, $x,y\in\mathbb R^n$, $t\in\mathbb R$, we write the homogeneous norm, for instance, as 
\begin{equation} \label{hcomplexnorm}
N_G(x,y,t)=((|x|^2+|y|^2)^2+16t^2)^{1/4}.
\end{equation}


\subsection{The $p$-module of $\Gamma(R_{ab};S_a,S_b)$ on polarizable groups}


Let $B(g,r)$, $g\in G$, be an open ball with respect to the norm $N_G$ in a polarizable Carnot group $G$, and let $S_r=\partial B(g,r)$ be the boundary of $B(g,r)$. We want to present the extremal function and the extremal family of curves for the condenser $(R_{ab};S_a,S_b)$ in the problem of the $p$-module.

One of the results in~\cite{KorReim87} by Kor\'{a}nyi and Reimann is the precise value of $M_p(\Gamma((R_{ab};S_a,S_b)))$ in $\heis^1$, where $p=4$ is the homogeneous dimension of~$\heis^1$. The $p$-module of the family of curves for the spherical ring domain on $H$-type groups and polarizable Carnot groups,  in terms of $p$-capacity is given in~\cite{balogh2001,capogna1996}. It is 
known, that  the value of $p$-capacity on the Carnot groups, see Definition~\ref{def:p_capac}, and  of the $p$-module of the family of curves connecting $S_a$ and $S_b$ in the spherical ring domain coincide, see~\cite{markina2003}. We present here brief calculations of $M_p(\Gamma((R_{ab};S_a,S_b)))$ on a polarizable Carnot group, for the completeness.

We recall that if $c\colon [a,b]\to G$ is an absolutely continuous curve in a Carnot group $G$, which is not horizontal for some 
open subinterval $I\subset [a,b]$, then it is non-rectifiable. It was shown in~\cite{Pansu} that even if $c$ is only continuous and rectifiable, then the tangent vector $\dot c(s)$ exists and it is horizontal for almost all $s\in [a,b]$.
Thus, when computing the $p$-module of a family of curves, we can restrict 
ourselves to horizontal curves, because the $p$-module of a family of non-rectifiable curves vanishes~\cite{Fug}. 
Note also that  any system of curves for $0<p<1$ has vanishing $p$-module~\cite{Fug}.

Let $\phi$ be a solution to the Cauchy problem~\eqref{radialflow} satisfying the initial data $\phi(1,\xi)=\xi$, $\xi\in S_1$. The presence of the horizontal flow $\phi(\cdot,\xi)\colon (0,\infty)\to G$, allows us to write the integral over $G$ in terms of spherical coordinates. Namely, the following proposition holds.

\begin{proposition}\label{prop:sph_int}\cite{balogh2001}
Let $G$ be a polarizable Carnot group. There exists a unique Radon measure $dv$ on $S_1\setminus\mathcal Z$, such that the integration formula
\begin{equation} \label{polardecompG}
\int_G f(g)d\mathbf{g}=\int_{S_1\setminus\mathcal{Z}}\int_0^\infty f(\phi(s,\xi))s^{Q-1}ds\,dv(\xi)
\end{equation}
is valid for all $f\in L^1(G)$, where $d\mathbf{g}$ denotes the Haar measure on $G$.
\end{proposition}

Observe that the only information one needs to carry on the construction of spherical coordinates and forthcoming calculation of $M_p(\Gamma((R_{ab};S_a,S_b)))$ on a polarizable Carnot group is the
existence of homogeneous norm $N_G$. 

\begin{theorem}~\label{ThmodulusCarnot}
 Let $G$ be a polarazible Carnot group of Hausdorff dimension $Q$ with a homogeneous norm $N_G$ associated to Folland's 
solution to the Kohn sub-Laplacian. Let $\Gamma=\Gamma(R_{ab};S_a,S_b)$ be a family of horizontal locally rectifiable curves connecting the boundaries $S_a$ and $S_b$ in $R_{ab}$.
Then, for $p>1$,
$$
M_p(\Gamma)=C_{S_1}(p)C_{ab}^{1-p}(p,Q),
$$
where 
$$
 C_{S_1}(p)=\int_{S_1\setminus\mathcal{Z}}\lambda^{p}(\xi)\,dv(\xi),
\quad\text{and}\quad 
C_{ab}(p,Q):=\int_a^{b}s^{\frac{1-Q}{p-1}}\,ds.
$$
\end{theorem}

\begin{proof}
Let us use the integration in spherical coordinates~\eqref{polardecompG} in order to calculate the module 
of $\Gamma$.
For all admissible functions $\varrho$ we have
\begin{equation*}
 1\leqslant\left(\int_{\phi(\cdot,\xi)}\varrho\right)^p=\left(\int_a^{b}\varrho(\phi(s,\xi))\lambda(\xi)^{-1}ds\right)^p\quad\Longrightarrow\quad \lambda(\xi)\leq\int_{\phi(\cdot,\xi)}\varrho,
\end{equation*}
where  $\lambda(\xi)^{-1}=\|\frac{d\phi(s,\xi)}{ds}\|_0=\frac{1}{\|\nabla_0N_G(\phi(s,\xi))\|_0}$, $\xi\in S_1$. H\"{o}lder's inequality implies
\begin{multline*}
\lambda(\xi)^{p}\leqslant\left(\int_a^{b}\left(\varrho s^{\frac{Q-1}{p}}\right)s^{-\frac{Q-1}{p}}ds\right)^p\leqslant
\left[\left(\int_a^{b}\varrho^p s^{Q-1} ds\right)^{\frac{1}{p}}\left(\int_a^{b}s^{\frac{1-Q}{p-1}} ds\right)^{\frac{p-1}{p}}\right]^p\\
=\left(\int_a^{b}\varrho^p s^{Q-1} ds\right)\left(\int_a^{b}s^{\frac{1-Q}{p-1}} ds\right)^{p-1}.
\end{multline*}
Therefore
\begin{equation*}
 \int_a^{b}\varrho^p(\phi(s,\xi))s^{Q-1} ds\geqslant C_{ab}(p,Q)^{1-p}\lambda(\xi)^{p},	
\end{equation*}
and,
\begin{eqnarray*}
\int_{R_{ab}}\varrho^p(g)\,d\mathbf{g} & = & 
\int_{S_1\setminus\mathcal{Z}}\int_a^b \varrho^p(\phi(s,\xi))s^{Q-1}ds\,dv(\xi)
\\
& \geq &
C_{ab}(p,Q)^{1-p}\int_{S_1\setminus\mathcal{Z}}\lambda(\xi)^{p}\,dv(\xi)=C_{ab}(p,Q)^{1-p}C_{S_1}(p).	
\end{eqnarray*}
If we denote by $\Gamma_0$ the family of curves formed by the radial flow $\phi$, then $\Gamma_0$ is a subfamily of the family $\Gamma$. Taking infimum over admissible functions we obtain that
\begin{equation}
 M_p(\Gamma)\geq M_p(\Gamma_0)\geq C_{ab}(p,Q)^{1-p}C_{S_1}(p).
\end{equation}

To find an estimation from above for $ M_p(\Gamma)$ we present the extremal function on which this estimate is attained. It is given by
\begin{equation}\label{eq:extr_curve}
\varrho_0=
\begin{cases}\Big((\tau+1)C_{ab}(p,Q)\Big)^{-1}\,\|\nabla_0 \left(N_G^{\tau+1}\right)\|_0,\quad\tau+1=\frac{p-Q}{p-1},\quad &p\neq Q,
\\
C_{ab}(p,Q)^{-1}\,\|\nabla_0 (\log N_G)\|_0,\quad &p=Q.
\end{cases}
\end{equation}
Considering $\varrho_0$ along the flow $\phi(s,\xi)$ of radial curves we obtain
\begin{equation*}
\varrho_0(\phi(s,\xi))=C_{ab}(p,Q)^{-1}\,s^\tau\lambda(\xi).
\end{equation*}
Using the integration in spherical coordinates~\eqref{polardecompG} we calculate 
\begin{multline*}
 \int_{R_{ab}}\varrho_0^p\,d\mathbf{g}=\int_{S_1\setminus\mathcal{Z}}\int_a^b \varrho_0^p(\phi(s,\xi))s^{Q-1}ds\,dv(\xi)\\
=C_{ab}(p,Q)^{-p}\int_{S_1\setminus\mathcal{Z}}\int_a^b\,s^{\frac{p(1-Q)}{p-1}}\lambda(\xi)^{p}s^{Q-1}ds\,dv(\xi)\\
=C_{ab}(p,Q)^{-p}\int_{S_1\setminus\mathcal{Z}}\lambda(\xi)^{p}dv(\xi)\int_a^bs^{\frac{1-Q}{p-1}}ds
=C_{ab}(p,Q)^{1-p}C_{S_1}(p).
\end{multline*}
The function $\varrho_0$ is admissible for $\Gamma$ as it will be shown in the next  Lemma~\ref{Admisfunc}. Finally, we have  
\begin{equation*}
 M_p(\Gamma)\leqslant\int_{R_{ab}}\varrho_0^p\,d\mathbf{g}=C_{ab}(p,Q)^{1-p}C_{S_1}(p).
\end{equation*}
\end{proof}

\begin{lemma} \label{Admisfunc}
For all $p>1$, the function $\varrho_0$, defined by~\eqref{eq:extr_curve} is admissible for the module $M_p(\Gamma)$ of a
the family of curves connecting $S_a$ to $S_b$ in the spherical ring domain $R_{ab}$.
\end{lemma}
\begin{proof}
Let $\gamma\colon [0,l_{\gamma}]\to G$ be any curve in $\Gamma$ parametrized by arc-length: $\|\dot\gamma\|_0=1$, and such that $a=N(\gamma(0))$ and 
$b=N(\gamma(l_{\gamma}))$. Then, by the Schwarz inequality we have 
$\langle\nabla_0 N,\dot{\gamma}(s)\rangle_0\leqslant\|\nabla_0 N\|_0$ for almost all $s\in[0,l_{\gamma}]$. It follows that for $p\neq Q$, 
\begin{eqnarray*}
 \int_{\gamma}\varrho_0 
 & = & \int_0^{l_{\gamma}}\varrho_0(\gamma(s))ds
 \\
 & = &
\Bigl((\tau+1)C_{ab}(p,Q)\Bigr)^{-1}\int_0^{l_{\gamma}}(\tau+1)N^\tau(\gamma(s))\|\nabla_0 N(\gamma(s))\|_0\,ds
\\
& \geq & \Bigl((\tau+1)C_{ab}(p,Q)\Bigr)^{-1}\int_0^{l_{\gamma}} (\tau+1)N^\tau(\gamma(s))\langle\nabla_0 N(\gamma(s)),\dot{\gamma}(s)\rangle ds
\\
& = & \Bigl((\tau+1)C_{ab}(p,Q)\Bigr)^{-1}\int_0^{l_{\gamma}}\frac{d}{ds}N^{\tau+1}(\gamma(s))ds.
\end{eqnarray*}
Since $N^{\tau+1}(\gamma)\colon [0,l_{\gamma}]\to R_{ab}$ is absolutely continuous, the Fundamental Theorem of Calculus results in
\begin{equation*}
\int_{\gamma}\varrho_0\geq \Bigl((\tau+1)\int_a^bs^\tau\,ds\Bigr)^{-1}\Bigl(N^{\tau+1}(\gamma(l_{\gamma}))-N^{\tau+1}(\gamma(0))\Bigr)=1.
\end{equation*}
For $p=Q$, we obtain
\begin{eqnarray*}
 \int_{\gamma}\varrho_0
 & = & \int_0^{l_{\gamma}}\varrho_0(\gamma(s))ds
  = 
\Bigl(C_{ab}(p,Q)\Bigr)^{-1}\int_0^{l_{\gamma}}N^{-1}(\gamma(s))\|\nabla_0 N(\gamma(s))\|_0\,ds
\\
& \geq & 
\Bigl(C_{ab}(p,Q)\Bigr)^{-1}\int_0^{l_{\gamma}}N^{-1}(\gamma(s))\langle\nabla_0 N(\gamma(s)),\dot{\gamma}(s)\rangle ds
\\
& = & 
\Bigl(\int_a^b\frac{ds}{s}\Bigr)^{-1}\int_0^{l_{\gamma}}\frac{d}{ds}\log N(\gamma(s))ds=1.
\end{eqnarray*}
 \end{proof}

\begin{corollary}
The family of radial curves $\Gamma_0$ satisfying~\eqref{radialflow} is the extremal family for the module $R_{ab}$ of the spherical ring domain $R_{ab}$ on polarizable Carnot groups. The function $\varrho_0$ given by~\eqref{eq:extr_curve} is the extremal function. Moreover, calculating the integral $C_{ab}(p,Q)^{1-p}$, we obtain
\begin{equation*} 
M_p(\Gamma)=
\begin{cases}
C_{S_1}(p)\Bigl(\frac{|p-Q|}{p-1}\Bigr)^{p-1}\Bigl|(b^{\frac{p-Q}{p-1}}-a^{\frac{p-Q}{p-1}})\Bigr|^{1-p},\quad & p\neq Q,
\\
C_{S_1}(p)\Bigl(\log\frac{b}{a}\Bigr)^{1-Q},& p=Q.
\end{cases}
\end{equation*}
\end{corollary}

As it was mentioned,  the $H$-type groups are polarizable~\cite[Proposition 5.6]{balogh2001}, and the form of the homogeneous norm is given by~\eqref{eq:normHtype}. This allows us to calculate precisely the value of the constant $C_{S_1}(p)$:  
\begin{equation*} 
 C_{S_1}(p)=\int_{S_1\setminus\mathcal{Z}}\lambda(\xi)^{p}\,dv(\xi)=\frac{2\pi^{k+l/2}\,\Gamma\left(\frac{k+p}{4}\right)}
{4^l\Gamma\left(\frac{k}{2}\right)\Gamma\left(\frac{k+2l+p}{4}\right)},
\end{equation*}
where $k=\dim V_1$, $l=\dim V_2$.
The details can be found in~\cite{balogh2001}.


\subsection{The $p$-module of a family of separating sets in $R_{ab}$}


We recall the definition of separating sets, given in Section~\ref{sep_sets}. Unfortunately, some technical difficulties do not allow us to consider the separating sets in full generality in this section. 
Let $\Sigma=\Sigma(R_{ab}; S_a,S_b)$ denote the class of all countably $H^{Q-1}$-rectifiable sets that separate $S_a$ from $S_b$ in $R_{ab}\subset G$. With every $\sigma\subset\Sigma$ we associate a complete measure $\mu$ in the following way: for every Hausdorff $H^{Q-1}$-measurable set $A\subset G$ define 
\[
 \mu(A)=\mathcal{H}_H^{Q-1}(A\cap\sigma\cap R_{ab}),
\]
where $Q$ is the Hausdorff dimension of the group $G$. Let $E$ denote the family of such measures associated with $\Sigma$. Let us describe this measures on spheres $S_s$ in details. The integration formula~\eqref{polardecompG} implies that the volume element $d\mathbf g$ along the flow defined by $\phi$ can be written as 
$$
d\mathbf g=s^{Q-1}dsdv(\xi)=s^{Q-1}\Big\|\frac{\partial\phi}{\partial s}\Big\|_0^{-1}dv(\xi)\Big\|\frac{\partial\phi}{\partial s}\Big\|_0ds=s^{Q-1}\lambda(\xi)dv(\xi)\lambda(\xi)^{-1}ds.
$$
Therefore, the measure $dS_1(\xi)=\lambda(\xi)dv(\xi)$, $\xi\in S_1$ is absolutely continuous with respect to the Radon measure $dv(\xi)$ and represents an $H^{Q-1}$ dimensional surface measure on the unit sphere $S_1$. The element of the surface area on the sphere $S_{s}$ of radius $s$ is given by $dS_s=s^{Q-1}\lambda(\xi)dv(\xi)$. The part $d\phi=\lambda(\xi)^{-1}ds$ defines the element of  length of the curve $\phi(\cdot,\xi)$. In the case $G=\mathbb R^k$, we obtain that $\lambda(\xi)\equiv 1$, and $dS_1=dv$ is the usual surface element on the unit sphere. 

\begin{theorem}
\label{th:ModSerface}
 Let $G$ be a polarazible Carnot group of Hausdorff dimension $Q$ with the homogeneous norm $N_G$ associated to Folland's 
solution to the Kohn sub-Laplacian. Let $E$ be the family of measures associated with sets separating $S_a$ and $S_b$ in $R_{ab}$. Then for $q>1$ we obtain
\begin{equation*} 
M_q(E)=K_{ab}(q,Q)K_{S_1}^{1-q}(q),
\end{equation*}
 where 
\begin{equation}\label{eq:constants2} 
K_{S_1}(q)=\int_{S_1\setminus\mathcal{Z}}\lambda^{\frac{q}{q-1}}(\xi)\,dv(\xi),\quad
K_{ab}(q,Q)=\int_a^{b}s^{(1-q)(Q-1)}ds.
\end{equation}
\end{theorem}
\begin{proof}

Let $\rho$ be an admissible function for the family $E$. Then, for any sphere $S_s$, $a<s<b$ we have 
\begin{eqnarray*}
1 & \leq & \Big(\int_{S_s\setminus{\mathcal Z}}\rho(\phi(s,\xi))\, dS_s\Big)^q=\Big(\int_{S_1\setminus{\mathcal Z}}s^{Q-1}\rho(\phi(s,\xi))\lambda(\xi)dv(\xi)\Big)^q
\\
& \leq & s^{q(Q-1)}\Big(\int_{S_1\setminus{\mathcal Z}}\rho^q(\phi(s,\xi))dv(\xi)\Big)\Big(\int_{S_1\setminus{\mathcal Z}}\lambda^{\frac{q}{q-1}}(\xi)dv(\xi)\Big)^{q-1}.
\end{eqnarray*}
Thus,
$$
\int_{S_1\setminus{\mathcal Z}}\rho^q(\phi(s,\xi))dv(\xi)\geq s^{-q(Q-1)}\Big(\int_{S_1\setminus{\mathcal Z}}\lambda^{\frac{q}{q-1}}(\xi)dv(\xi)\Big)^{1-q}.
$$
Then, we arrive at the inequality
$$
\int_{R_{ab}}\rho^qd\mathbf g=\int_a^bs^{Q-1}ds\int_{S_1\setminus\mathcal Z}\rho^qdv\geq \int_a^{b}s^{(1-q)(Q-1)}\Big(\int_{S_1\setminus\mathcal Z}\lambda^{\frac{q}{q-1}}(\xi)dv(\xi)\Big)^{1-q}.
$$
Making use of notations~\eqref{eq:constants2}, we come to a lower bound for the module $M_q(E)$ of the family of separating sets
$$
M_q(E)\geq M_q(E_0)\geq K_{ab}(q,Q)K_{S_1}(q)^{1-q},
$$
where $E_0$ is the family of measures associated with the spheres $S_s=\{g\in G\mid N_G(g)=s\}$ for $a<s<b$ which separate the boundaries of the spherical ring domain $R_{ab}$.

Now we turn to the estimation of $M_q(E)$ from above. The extremal function in this case is given by the following expression
\begin{equation}\label{eq:extr_func_surf}
\rho_0(g)=
\begin{cases}
(\tau+1)^{\frac{1}{1-q}}K_{S_1}^{-1}(q)\|\nabla_0(N_{G}^{\tau+1}(g))\|_0^{\frac{1}{q-1}}, \tau=(q-1)(1-Q),& q\neq \frac{Q}{Q-1},
\\
K_{S_1}^{-1}(q)\|\nabla_0(\log N_G(g))\|_0^{\frac{1}{q-1}}, & q=\frac{Q}{Q-1}.
\end{cases}
\end{equation}
Restricting the value of $\rho_0$ to the sphere $N_G(g)=s$ we conclude that
$$
\rho_0(\varphi(s,\xi))=\rho_0(g)=K_{S_1}^{-1}(q)s^{1-Q}\lambda^{\frac{1}{q-1}}(\xi).
$$ 
Thus,
$$
\int_{R_{ab}}\rho_0^qd\mathbf g=K_{S_1}^{-q}(q)\int_a^bs^{Q-1+q(1-Q)}ds\int_{S_1\setminus\mathcal Z}\lambda^{\frac{q}{q-1}}dv(\xi)=K_{ab}(q,Q)K_{S_1}^{1-q}(q).
$$
The function $\rho_0$ is admissible for the family of separating sets as it will be proved in Subsection~\ref{subsec:adm_mod_surf}. Finally, taking the infimum over the admissible functions, we obtain
$$
M_q(E)\leq\int_{R_{ab}}\rho_0^qd\mathbf g=K_{ab}(q,Q)K_{S_1}^{1-q}(q).
$$
This finishes the proof.
\end{proof}

\begin{corollary}
The family of measures $E_0$ associated with the spheres $\Sigma_0=\{S_s,\ a<s<b\}$ is the extremal family for the module $M_q(E)$ of   sets $\Sigma$ separating the spheres $S_a$ and $S_b$ in the spherical ring domain $R_{ab}$ on polarizable Carnot groups. The function $\rho_0$ given by~\eqref{eq:extr_func_surf} is  extremal. In particular, $\int_{S_s}\rho_0\,dS_s=1$ for any sphere $S_s$, $a<s<b$.
\end{corollary}

Let us observe the following relations that reveal Theorems~\ref{ThmodulusCarnot} and~\ref{th:ModSerface}.

\begin{corollary}\label{cor:1-2} For $\frac{1}{p}+\frac{1}{q}=1$, Theorems~\ref{ThmodulusCarnot} and~\ref{th:ModSerface} imply 
\begin{itemize}
\item[1.] $K_{ab}(q,Q)=C_{ab}(p,Q)$,
\item[2.] $K_{S_1}(q)=C_{S_1}(p)$,
\item[3.] $M_p^{\frac{1}{p}}(\Gamma)M_q^{\frac{1}{q}}(\Sigma)=1$,
\item[4.] $\rho_0=C_{ab}^{p-1}(p,Q)C_{S_1}^{-1}(p)\varrho_0^{p-1}$, $p\neq Q$ 
\end{itemize}
\end{corollary}


\subsubsection{Relations between $M_p(\Gamma)$, $M_q(E)$, and the capacity $\capac_p(\R_{ab})$.}

Before we proceed to show that the function $\rho_0$ is admissible for the family $E$, we review the relations between $M_p(\Gamma)$, $M_q(E)$, and the capacity $\capac_p(\R_{ab})$.  

\begin{definition}\label{def:p_capac}
Let $\Omega$ be a domain in $G$, and let $D_0,D_1$ be two disjoint compacts in the closure $\overline \Omega$ of  $\Omega$. A function $u\in W^{1,p}(\Omega)$, such that $u\vert_{D_0}=0$ and $u\vert_{D_1}=1$, is called admissible for the condenser $(\Omega; D_0,D_1)$. The value 
$$
\capac_p(\Omega;D_0,D_1)=\inf\int_{\Omega}\|\nabla_0 u\|_0^p\,dx,
$$ 
is called a $p$-capacity of the condenser $(\Omega; D_0,D_1)$, where the infimum is taken over all admissible functions $u$.
\end{definition}

L\"owner introduced $3$-capacity in $\mathbb R^3$ in~\cite{Loew}, and showed that $\capac_3(\Omega)>0$.  Gehring~\cite{Geh62} proved that the L\"owner $3$-capacity (or conformal capacity) for a ring domain in $\mathbb R^3$, coincides with the module $M_{3}(\Gamma)$  of a family of curves, which was calculated  by V\"ais\"al\"a earlier in~\cite{va61}, and that it is also equal to the module $M_{3/2}(E)^{-2}$ of the family of surface measures on compact piecewise smooth surfaces $\Sigma$ separating $D_0$ and $D_1$ in a bounded domain $\Omega\subset\mathbb R^3$. The latter notion was used by \v{S}abat~\cite{Shabat} in his study of quasiconformal maps in $\mathbb R^3$. The restriction to smooth surfaces was relaxed in~\cite{Kriv} provided that admissible functions behave sufficiently nice. Later in 1966-68, Zimmer showed that the module $M_n(\Gamma)$ of a family of curves connecting $D_0$ and $D_1$ in a bounded domain $\Omega\subset\mathbb R^n$ is in the following relation with  the module $M_{\frac{n}{n-1}}(E)$ of the family of measures associated with the sets separating $D_0$ and $D_1$~\cite{zim66}:
\begin{equation}\label{eq:mm}
\Big(M_n(\Gamma)\Big)^{\frac{1}{n}}\Big(M_{\frac{n}{n-1}}(E)\Big)^{\frac{n-1}{n}}=1.
\end{equation}
In order to relax the conditions on admissible functions, the method of symmetrisaition in~\cite{Geh61} and surface-theoretical approximation theorems permit to consider general separating sets, see~\cite{Fed}. Shlyk showed in~\cite{Shlyk90}, that for a rather general condenser $(\Omega;D_0,D_1)$ in $\mathbb R^n$, the equality~\eqref{eq:mm} can be extended as follows
\begin{equation}\label{eq:mmp}
\Big(M_p(\Gamma)\Big)^{\frac{1}{p}}\Big(M_{q}(E)\Big)^{\frac{1}{q}}=1,\quad\frac{1}{p}+\frac{1}{q}=1.
\end{equation}
For further interesting generalizations for modules in $\mathbb R^n$ see~\cite{AikOht}. Some extensions to the Carnot groups can be found in~\cite{mar2003,mar2004,mar2005}, and for arbitrary metric measure spaces for instance in~\cite{Shan}.

Zimmer proved \cite{zim68, zim69}, that the capacity $\capac_p(\Omega;D_0,D_1)$ coincides with the module $M_p(\Gamma(\Omega;D_0,D_1))$ of a family of curves connecting $D_0$ and $D_1$ in $\Omega$, where the domain $\Omega\subset\mathbb R^n$ is assumed to be bounded. Hesse~\cite{Hesse} extended his result to unbounded domains. In particular, he showed that the set of admissible functions for the $p$-module of a family of curves can be restricted from non-negative Borel measurable functions in $\mathbb R^n$ to lower semicontinuous $L_p$-functions  in $\mathbb R^n$, which are continuous in $\Omega$, provided that $D_0\cup D_1\subset\Omega$. Shlyk~\cite{Shlyk93}  generalized the result of Hesse from a connected open set (domain) $\Omega$ to an arbitrary open set in $\overline{\mathbb R}^n$.

In general, the relation between the admissible function $u$ for the $p$-capacity of a condenser $(\Omega; D_0, D_1)$, and the admissible function $\rho$ for the $p$-module of a family of curves connecting $D_0$ and $D_1$ is as follows. Let $\rho$ be an admissible function for the family of curves connecting $D_0$ and $D_1$. Then the function $u(x)=\min\{1,\inf\int_{\beta_x}\rho\}$ is admissible for the $p$-capacity of the condenser $(\Omega; D_0, D_1)$, where the infimum is taken over all locally rectifiable curves $\beta_x$ in $\Omega$ connecting $D_0$ and the point $x\in \Omega$.  Moreover, 
$$
|\nabla u|\leq \rho\quad\text{almost everywhere in $\Omega$}.
$$
This immediately implies the inequality
$$
\capac_p(\Omega; D_0, D_1)\leq\int_{\mathbb R^n}|\nabla u|^p\,dx\leq \int_{\mathbb R^n}\rho^p\,dx\leq M_p(\Gamma),
$$
by taking infimum over all admissible functions $\rho$ for the $p$-module. On the other hand, if $u$ is an admissible $W^{1,p}$-function for the $p$-capacity of $(\Omega; D_0, D_1)$, then 
$$
\rho(x)=
\begin{cases}
|\nabla u(x)|,\quad & x\in\Omega,
\\
0, & x\in\mathbb R^n\setminus\Omega
\end{cases}
$$
is an admissible function for the module $M_p(\Gamma)$ of the family of curves connecting $D_0$ and $D_1$, that implies the inequality
$$
M_p(\Gamma)\leq\int_{\mathbb R^n}\rho^p\,dx=\int_{\mathbb R^n}|\nabla u|^p\,dx\leq \capac_p(\Omega; D_0, D_1)
$$
upon taking infimum over all admissible functions $u$ for the $p$-capacity.

Let us also mention a relation between the extremal functions $\varrho_0$, $\rho_0$ for the modules $M_p(\Gamma)$ and $M_q(E)$, and the extremal function $u$ for the $p$-capacity of the condenser $(R_{ab};S_a,S_b)$. It is well known that the variational equation for the problem of finding the $p$-capacity on a polarizable Carnot group $G$ (and particularly in $\mathbb R^n$) is the $p$-sub-Laplacian equation, and the extremal function $u$ for the $p$-capacity is a solution to the $p$-sub-Laplace equation in $G$ with  prescribed boundary values on $D_0$ and $D_1$. It was shown~\cite{balogh2001}, that the function 
$$
\tilde u(g)=
\begin{cases}
c_pN_G^{\tau+1},\quad \tau+1=\frac{p-Q}{p-1},& \quad\text{for}\quad p\neq Q,
\\
c_Q\log N_G, &\quad\text{for}\quad p= Q,
\end{cases}
$$
is a fundamental solution to the $p$-sub-Laplacian equation on $G$ for some appropriate choice of constants, see~\cite{capogna1996} for an analogous result on $H$-type groups.  One can easily check that 
$$
u(g)=
\frac{N_G^{\tau+1}(g)-a^{\tau+1}}{b^{\tau+1}-a^{\tau+1}},\quad g\in R_{ab} 
$$
is extremal for the $p$-capacity of $(R_{ab};S_a,S_b)$. 


\subsection{Admissibility of $\rho_0$ for $M_q(E)$}\label{subsec:adm_mod_surf}


In this section we will show that the function $\rho_0$ defined in~\eqref{eq:extr_func_surf} is admissible for a system $E$ of Hausdorff measures $H^{Q-1}$ associated with a family $\Sigma$ of countably $H^{Q-1}$-rectifiable sets separating $S_a$ and $S_b$ in $R_{ab}$. The core idea of the proof is to show that if $u$ is an extremal function for the $p$-capacity of $(R_{ab};S_a,S_b)$, then $\|\nabla_0u\|_0^{p-1}$ is an admissible function for $M_q(E)$ with $\frac{1}{p}+\frac{1}{q}=1$. This method 
goes back to Gehring~\cite{Geh62} who proved a similar result for $\mathbb R^3$, which was extended for the $n$-capacities and $n$-modules by  Ziemer~\cite{zim66}  in $\mathbb R^n$. Later, Shlyk~\cite{Shlyk90} generalized the proof to $\mathbb R^n$ for arbitrary values of $p\neq n$. The same result was implicitly presented in~\cite{AikOht} for $\mathbb R^n$, and in~\cite{markina2003} for arbitrary Carnot groups. Here we want to follow the ideas of Gehring~\cite{Geh62}. We emphasize that in $\mathbb R^n$ the result was obtained for arbitrary system of separating sets. The lack of approximations theorems, such as, for instance~\cite[Theorem 2.4.2]{zim66}, does not allow us to extend the proof of Theorem~\ref{th:adm} to arbitrary separating sets on polarizable groups. Our main goal is to show the following theorem.

\begin{theorem}\label{th:adm}
Let $\Sigma$ be a family of countably $H^{Q-1}$-rectifiable sets separating $S_a$ and $S_b$ in $R_{ab}$, $\sigma\in\Sigma$, and let $u$ be an extremal function for the $p$-capacity of $(R_{ab},S_a,S_b)$. Let $E$ be a family of $(Q-1)$-Hausdorff measures $H^{Q-1}$ associated with $\Sigma$. Then the integral
$\int_{\sigma}\|\nabla_0u(g)\|_0^{p-1}\,dH^{Q-1}(g)$ exists for $M_q(E)$-almost all measures from $E$ and 
\begin{equation}\label{eq:adm}
\int_{\sigma}\|\nabla_0u(g)\|_0^{p-1}\,dH^{Q-1}(g)\geq \capac_p(R_{ab};S_a,S_b),\quad \frac{1}{p}+\frac{1}{q}=1.
\end{equation}
\end{theorem}

The proof is forestalled by two lemmas. Before we formulate the statement of the first lemma let us describe some constructions which we will use. Let $\sigma\in \Sigma$, and let $\beta>0$ be such that $\beta<\dist(\sigma,\partial R_{ab})$. We denote by $\sigma(\beta)=\{g\in R_{ab}\mid\ \dist(\g,\sigma)<\beta\}$ and by $d(g)=\dist(g,\sigma)$. Here the distance is understood as  $\dist(g,\eta)=N_G(\eta^{-1}g)$, $g,\eta\in G$. By construction, the norm $N_G=u_2^{\frac{1}{2-Q}}$ is smooth away from the identity of $G$ due to the smoothness of the solution $u_2$. This guarantees that the function $d$ is at least Lipschitz in~$G$.

\begin{lemma}\label{lem:1}
Let $u$ be an extremal function for the $p$-capacity of the condenser $(R_{ab}; S_a,S_b)$, and let $\sigma\in\Sigma$. Then
$$
\int_{\sigma(\beta)}\|\nabla_0u(g)\|_0^{p-1}\|\nabla_0d(g)\|_0\,d{\mathbf g}\geq 2\beta\capac_p(R_{ab}; S_a,S_b).
$$
\end{lemma}
\begin{proof} Denote by $\overline R_{ab}$ the closure of $R_{ab}$, and by $A^c$ the complement to $A$ in $G$.
Let $F_0$ be a component of $\sigma^c\cap \overline R_{ab}$  containing $S_a$, and  let $F_1$ be a component of $\sigma^c\cap \overline R_{ab}$ containing $S_b$. Let
$$
E_k=\{g\in \overline R_{ab}\mid\ 0<\dist(g,F^c_k)<\beta\},\quad k=0,1. 
$$
Then $E_k\subset F_k$, $E_0\cup E_1\subset \sigma(\beta)$, and it is sufficient to show that 
$$
\int_{E_k}\|\nabla_0u(g)\|_0^{p-1}\|\nabla_0d(g)\|_0\,d{\mathbf g}\geq \beta\capac_p(R_{ab};S_a,S_b),\quad k=0,1.
$$

We  focus ourselves only on the case $k=0$. The case $k=1$ is treated analogously. Define
$$
v(g)=\min\{\beta,\dist(g,F^c_0)\}=
\begin{cases}
0,\qquad &\quad\text{if}\quad g\in F_0^c=F_1\cup\sigma
\\
\inf_{\eta\in F^c_0}\dist(g,\eta)&\quad\text{if}\quad  g\in E_0
\\
\beta &\quad\text{if}\quad g\in F_0\setminus E_0.
\end{cases}
$$
The function $v$ is  Lipschitz, from the class $L^p(R_{ab})$, and 
$$
\|\nabla_0 v(g)\|_0=
\begin{cases}
\|\nabla_0 d(g)\|_0>0 &\quad\text{almost everywhere in}\quad E_0,
\\
0&\quad\text{ in}\quad  R_{ab}\setminus E_0.
\end{cases}
$$
Thus, the function $w=v-\beta u$ is almost everywhere differentiable and belongs to the class $L^p(R_{ab})$. We use $w$ as a test function on $R_{ab}$ and obtain
$$
0=\int_{R_{ab}}\|\nabla_0u(g)\|_0^{p-2}\langle\nabla_0u,\nabla_0w\rangle_{0}d\mathbf g=\int_{R_{ab}}\|\nabla_0u(g)\|_0^{p-2}\langle\nabla_0u,\nabla_0v-\beta\nabla_0u\rangle_{0}d\mathbf g.
$$
This, together with the Cauchy-Schwartz inequality, implies
\begin{eqnarray*}
\int_{E_0}\|\nabla_0u(g)\|_0^{p-1}\|\nabla_0d(g)\|_0\,d{\mathbf g} & = &  \int_{R_{ab}}\|\nabla_0u(g)\|_0^{p-1}\|\nabla_0v(g)\|_0\,d{\mathbf g}
\\ &\geq & \int_{R_{ab}}\|\nabla_0u(g)\|_0^{p-2}\langle \nabla_0u(g),\nabla_0v(g)\rangle_0\,d{\mathbf g}
\\
& = & \beta \int_{R_{ab}}\|\nabla_0u(g)\|_0^{p}\,d{\mathbf g}
= \beta\capac_p(R_{ab};S_a,S_b).
\end{eqnarray*}
\end{proof}

If we were pass to the limit in $$\frac{1}{2\beta}\int_{\sigma(\beta)}\|\nabla_0u(g)\|_0^{p-1}\|\nabla_0d(g)\|\,d{\mathbf g}\geq\capac_p(R_{ab};S_a,S_b)$$ as $\beta\to 0$, we could finish the proof of Theorem~\ref{th:adm} at once. In order to show that the limit exists, we consider the sequences of continuous functions $f_r(g)$ converging to $\|\nabla_0u(g)\|_0^{p-1}$ $\mathbf g$-almost everywhere as $r\to 0$ and such that the limit $$\frac{1}{2\beta}\int_{\sigma(\beta)}f_r(g)\|\nabla_0d(g)\|\,d{\mathbf g}\to\int_{\sigma}f_r(g)\,d{H}^{Q-1}\quad\text{as}\quad \beta\to 0$$
exists.
We define the integral mean of $\|\nabla_0u(g)\|_0^{p-1}$ in the ball by
\begin{equation}\label{eq:average}
f_r(g)=\frac{1}{\mathbf g(B(g,r))}\int_{B(g,r)}\|\nabla_0u(\eta)\|_0^{p-1}\,d\mathbf g(\eta).
\end{equation}

We also recall the co-area formula for Carnot groups. Let $U$ be a domain in $G$. Let $f\in L^1(U)$ be a non-negative function, and let $v$ be a real valued Lipschitz function in $U$, see~\cite{Heinonen95,KarmVodopis,Magnani2005}. Then  
\begin{equation}\label{eq:co_area}
\int_{U}f(g)\|\nabla_0v(g)\|_0\,d\mathbf g(g)=\int_{-\infty}^{+\infty}\int_{v^{-1}(s)}f(\eta)dH^{Q-1}(\eta)ds.
\end{equation}
 
\begin{lemma}\label{lem:2}
The integral mean~\eqref{eq:average} satisfies the inequality
$$
\int_{\sigma}f_r(g)\,d{H}^{Q-1}\geq\capac_p(R_{ab};S_a,S_b),
$$
whenever $r<\dist(\sigma,\partial R_{ab})$.
\end{lemma}
\begin{proof}
If $H^{Q-1}(\sigma)=\infty$, then there is nothing to prove. Assume that $H^{Q-1}(\sigma)<\infty$.
Let $\beta,r$ be positive numbers such that $\beta+r<\dist(\sigma,\partial R_{ab})$. Let $L_\eta(\sigma)=\eta\sigma$ be a left translation of the set $\sigma$ by an element $\eta\in G$. Then, changing variables and using  Fubini's theorem we come to 
\begin{eqnarray}\label{1part}
 &\ &\int_{\sigma(\beta)}f_r(g)\|\nabla_0d(g)\|_0\,d\mathbf g(g)
 \\
 &= & \frac{1}{\mathbf g(B(g,r))}\int_{B(0,r)}\,d\mathbf g(\eta)\int_{\sigma(\beta)}\|\nabla_0u(\eta g)\|_0^{p-1}\|\nabla_0d(g)\|_0 \,d\mathbf g(g).\nonumber
 \end{eqnarray}
Observe that $d(g)=\dist(g,\sigma)=\dist(\eta g,\eta\sigma)=\dist(\psi,\eta\sigma)$, with $\psi=\eta g$. Then making change of variables $\psi=\eta g$, we write the last integral in the form
\begin{eqnarray*}
& \ &\frac{1}{\mathbf g(B(g,r))}\int_{B(0,r)}\,d\mathbf g(\eta)\int_{\eta\sigma(\beta)}\|\nabla_0u(\psi)\|_0^{p-1}\|\nabla_0\dist(\psi,\eta\sigma)\|_0 \,d\mathbf g(\psi)
\\
& = &\frac{1}{\mathbf g(B(g,r))}\int_{B(0,r)}\,d\mathbf g(\eta)\int_{(\eta\sigma)(\beta)}\|\nabla_0u(\psi)\|_0^{p-1}\|\nabla_0d(\psi)\|_0 \,d\mathbf g(\psi)
\\
&  \geq & 2\beta\capac_p(R_{ab};S_a,S_b),
\end{eqnarray*}
where the last inequality follows from Lemma~\ref{lem:1}. Moreover, applying the co-area formula~\eqref{eq:co_area}, we obtain
$$
\int_{\sigma(\beta)}f_r(g)\|\nabla_0d(g)\|_0\,d\mathbf g(g)=\int_0^{\beta}\int_{d^{-1}(s)}f_r(\zeta)dH^{Q-1}(\zeta)ds.
$$
Let $F(s)$ denote the interior integral in the right-hand side,
$$
F(s)=\int_{d^{-1}(s)}f_r(\zeta)dH^{Q-1}(\zeta).
$$
The function $f_r$ is continuous on $G$, and the set $\sigma$ is countably rectifiable.  Then 
$$
F(s)=\int_{d^{-1}(s)}f_r(\zeta)dH^{Q-1}(\zeta)\to 2F(0)=2\int_{\sigma}f_r(\zeta)dH^{Q-1}(\zeta),\quad\text{as}\quad s\to 0.
$$
Thus gathering the above results we arrive at
\begin{eqnarray*}
\capac_p(R_{ab};S_a,S_b) &\leq & \lim\limits_{\beta\to 0}2\frac{1}{2\beta}\int_{\sigma(\beta)}f_r(g)\|\nabla_0d(g)\|_0d\mathbf g(g) 
\\
&= &\lim\limits_{\beta\to 0}\frac{1}{\beta}\int_{0}^\beta F(s)ds=\int_{\sigma}f_r(\zeta)\,H^{Q-1}(\zeta).
\end{eqnarray*}
\end{proof}

\noindent
{\it Proof of Theorem~\ref{th:adm}.} Let  $\sigma\in\Sigma$. We can assume that for any $r<r_0<\dist(\partial F_0,\partial R_{ab})$, the support of $f_r$ belongs to $R_{ab}$. Then, 
$$
f_r\to \|\nabla_0u\|_0^{p-1}\quad \mathbf g-\text{almost everywhere as}\quad r\to 0,
$$
and $\int_{R_{ab}}f_r^{q}\,d\mathbf g\leq \int_{R_{ab}}\|\nabla_0u\|_0^{p}\,d\mathbf g<\infty$. The Lebesgue dominated convergence theorem implies that $f_r$ converges to $ \|\nabla_0u\|_0^{p-1}$ in $L^q(R_{ab})$ as $r\to 0$. Therefore, there is a subsequence (that we will denote by the same symbol) $f_r$, such that
$$
\int_{\sigma}\Big|f_r-\|\nabla_0u\|_0^{p-1}\Big|\,dH^{Q-1}\to 0\quad\text{as}\quad r\to 0,
$$
for $M_q(E)$-almost all measures $\mu\in E$. Thus, the integral
$$\int_{\sigma}\|\nabla_0u\|_0^{p-1}\,dH^{Q-1}$$ exists and moreover inequality~\eqref{eq:adm} holds. 
\qed

\begin{corollary}
The function $\rho_0$ defined in~\eqref{eq:extr_func_surf} is admissible for the module $M_q(E)$ of the family of measures associated with sets separating $S_a$ and $S_b$ in $R_{ab}$.
\end{corollary}

\begin{proof}
As it was mentioned, the function 
$$
u(g)=\frac{N_G^{\tau+1}(g)-a^{\tau+1}}{b^{\tau+1}-a^{\tau+1}},\quad g\in G,
$$
is extremal for $\capac_p(R_{ab};S_a,S_b)$ and  
$
\varrho_0=\|\nabla_0 u\|_0
$ by~\eqref{eq:extr_curve}. Therefore,
$$
\capac_p(R_{ab};S_a,S_b)=\int_{R_{ab}}\|\nabla_0 u\|_0^p\,d\mathbf g=\int_{R_{ab}}\varrho_0^p\,d\mathbf g=M_p(\Gamma),
$$
where $\Gamma$ is the family of all locally rectifiable curves connecting $S_a$ and $S_b$.
Moreover, $\rho_0=C_{ab}^{p-1}(p,Q)C_{S_1}^{-1}(p)\|\nabla_0 u\|_0^{p-1}$ as shows~\eqref{eq:extr_func_surf}. Thus, for any $\sigma\in \Sigma$ we obtain
\begin{equation*}
\begin{split}
\int_{\sigma}\rho_0\,dH^{Q-1} 
& =  C_{ab}^{p-1}(p,Q)C_{S_1}^{-1}(p)\int_{\sigma}\|\nabla_0 u\|_0^{p-1}\,dH^{Q-1}
\\
& \geq   C_{ab}^{p-1}(p,Q)C_{S_1}^{-1}(p)\ \capac_p(R_{ab},S_a,S_b)
\\
& \geq   C_{ab}^{p-1}(p,Q)C_{S_1}^{-1}(p)\ M_p(\Gamma)=1.
\end{split}
\end{equation*}
\end{proof}

\subsection{Twisting map of the spherical ring in the Heisenberg group}

Let us consider the spherical ring $R_{1b}=\{(x_1,x_2,t)\colon 1\leq N_{\mathbb H}(x_1,x_2,t)\leq b\}$ in the Heisenberg group $\mathbb H=\mathbb H^1$  with respect to the homogeneous norm $N_{\mathbb H}$. The coordinates in $\mathbb H$ are $(x_1,x_2,t)$, or in the polar form $(\theta,\alpha, r)$, $(x_1,x_2,t)=G(\theta,\alpha, r)$, where $G\colon (\theta,\alpha, r)\to (x_1,x_2,t)$ is given by
\begin{equation}\label{eq:Heis_polar}
x_1=r\sqrt{\cos \alpha}\cos\theta,\quad x_2=r\sqrt{\cos \alpha}\sin\theta,\quad t=r^2\sin \alpha,
\end{equation}
 $\theta\in [0,2\pi)$, $\alpha\in (-\pi/2,\pi/2)$, $r\in [1,b]$.
The horizontal vector fields in the polar form are
\[
X_1=\frac{\partial}{\partial x_1}+2x_2 \frac{\partial}{\partial t}=\sqrt{\cos\alpha}\left(\cos(\theta-\alpha)\frac{\partial}{\partial r}+\frac{2}{r}\sin(\theta-\alpha)\frac{\partial}{\partial \alpha}-\frac{\sin\theta}{r\,{\cos\alpha}}\frac{\partial}{\partial\theta}\right),
\]
 \[
X_2=\frac{\partial}{\partial x_2}-2x_1 \frac{\partial}{\partial t}=\sqrt{\cos\alpha}\left(\sin(\theta-\alpha)\frac{\partial}{\partial r}-\frac{2}{r}\cos(\theta-\alpha)\frac{\partial}{\partial \alpha}+\frac{\cos\theta}{r\,{\cos\alpha}}\frac{\partial}{\partial\theta}\right).
\]
The horizontal norm of the horizontal gradient of a smooth function $f(\theta,\alpha, r)$ is calculated as $\|\nabla_0f\|_0=(X^2_1(f)+X_2^2(f))^{1/2}$.
 The $H^{3}$-Hausdorff measure element on the sphere  $S_r\setminus \mathcal Z$ is $d\omega_r=r^3\sqrt{\cos\alpha }\,d\alpha d\theta$.  
In particular, the area of the unit sphere $S_1$ is calculated as
\[
\text{Area}(S_1)=4\sqrt{2\pi}\,\,\Gamma^2\left(\frac{3}{4}\right),
\]
where Euler's $\Gamma$-function is
$
\Gamma \left( x \right) = \int_0^\infty {t^{x - 1} e^{-t} dt}$.
The radial flow $\phi(r,\theta,\alpha)$ on $\mathbb H$ orthogonal to the sphere $S_1\setminus \mathcal Z$ is a solution to the initial-value problem~\eqref{radialflow} in the particular case of the homogeneous Heisenberg norm $N_{\mathbb H}$,
given by
\begin{eqnarray*}
x_1(r) &=& r\sqrt{\cos\alpha}\cos\left(\theta-\tan\alpha\,\log r\right),\\
x_2(r) &=& r\sqrt{\cos\alpha}\sin\left(\theta-\tan\alpha\,\log r\right),\\
t(r) &=& r^2\sin\alpha,
\end{eqnarray*}
where $1\leq r\leq b$, and   $\theta\in [0,2\pi)$, $\alpha\in (-\pi/2,\pi/2)$ are fixed. 
The horizontal norm of $\dot{\phi}(r,\theta,\alpha)=\frac{\partial}{\partial r}\phi(r,\theta,\alpha)$  is $\|\dot{\phi}(r,\theta,\alpha)\|_0=\cos^{-1/2}\alpha$.
 
An analogue to  Theorem~\ref{Rodin1} can be formulated for the spherical ring domain in $\mathbb H$ as follows. Let $\Gamma_0$ denote the family of curves $\phi_{\theta\alpha}(\cdot)\colon[1,b]\to\mathbb H$ given by radial flow $\phi(r,\theta,\alpha)$ for every fixed $\alpha\in (-\pi/2,\pi/2)$, $\theta\in [0,2\pi)$. In order to preserve the horizontal nature of the families of curves we require from a smooth map $f\colon \mathbb H\to \mathbb H$ to be the contact map, that is a map whose differential  preserves the horizontal planes $\spn\{X_1(g),X_2(g)\}$ for all $g\in\mathbb H$, see, for instance,~\cite{KorReim85}.

 \begin{theorem}\label{Rodin33}
Let $f\colon \mathbb H\to \mathbb H$ be a $C^1$-smooth orientation preserving contact map, and let $c_{\theta\alpha}(r)=f({\phi}_{\theta\alpha}(r))$.   Set $1/p+1/q=1$, $p,q>1$, and 
\[
\ell(\theta,\alpha)=\int_1^{R}\left(\frac{\|\dot{c}_{\theta\alpha}\|_0}{J_fr^3\sqrt{\cos\alpha}}\right)^{q}J_fr^3\sqrt{\cos\alpha}\,dr,\quad \alpha\in [-\pi/2,\pi/2], \,\,\,\theta\in [0,2\pi).
\]
Then 
\[
\rho_0(y)=\frac{1}{\ell(\theta,\alpha)}\left(\frac{\|\dot{c}_{\theta\alpha}\|_0}{J_fr^3\sqrt{\cos\alpha}}\right)^{\frac{1}{p-1}}\circ f^{-1},\quad G(\theta,\alpha,r)\in R_{1b}, 
\]
$y=f(x_1,x_2,t)\in R'_{1b}=f(R_{1b})$, is the extremal function for the $p$-module $M_{p}(f(\Gamma_0))$  and moreover $M_{p}(f(\Gamma_0))=\int_{R'_{1b}}\rho_0^p\,dy=\int_{-\pi/2}^{\pi/2}\int_0^{2\pi} \ell^{1-p}\,d\theta d\alpha$.
\end{theorem}
\begin{proof}
Let us first observe that $\int_{{c}_{\theta\alpha}}\rho_0\,ds=1$, where $ds$ is the arc-length element defined with respect to the norm $N_{\mathbb H}$. Indeed,
\begin{eqnarray*}
\int_{c_{\theta\alpha}}\rho_0\,ds&=&\int_1^{b}(\rho_0\circ f)\|\dot{c}_{\theta\alpha}\|_0\,dr=\frac{1}{\ell}\int_1^{b}\left(\frac{\|\dot{c}_{\theta\alpha}\|_0}{J_fr^3\sqrt{\cos\alpha}}\right)^{\frac{1}{p-1}}\|\dot{c}_{\theta\alpha}\|\,dr\\
&=&\frac{1}{\ell}\int_1^{b}\left(\frac{\|\dot{c}_{\theta\alpha}\|_0}{J_fr^3\sqrt{\cos\alpha}}\right)^{q}J_fr^3\sqrt{\cos\alpha}\,dr=1,
\end{eqnarray*}
for all $\alpha\in (-\pi/2,\pi/2)$ and $\theta\in [0,2\pi)$. Therefore, $\rho_0$ is admissible for $f(\Gamma_0)$ and 
\begin{equation}\label{moddd1}
M_p(f(\Gamma_0))\leq \int_{R'_{1b}}\rho_0^p\,d\mathbf g.
\end{equation}
On the other hand,
for any $\rho$ admissible for $f(\Gamma_0)$ we have $\int_{c_{\theta\alpha}}\rho\, ds \geq 1$, and therefore,
\[
\int_{c_{\theta\alpha}}(\rho-\rho_0)\, ds\geq 0.
\]
This implies that
\[
\frac{1}{\ell^{p-1}(\theta,\alpha)}\int_1^b [(\rho-\rho_0)\circ f\,]\,\|\dot{c}_{\theta\alpha}\|_0\,dr\geq 0.
\]
Then
\[
\int_{S_1\setminus\mathcal Z} \int_1^b \left((\rho-\rho_0)\rho^{p-1}_0\circ f\right)\,J_fr^3\, dr\,dH^{3}\geq 0.
\]
Equivalently,
\[
\int_{R'_{1b}}\rho\rho_0^{p-1}\,d\mathbf g\geq \int_{R'_{1b}}\rho_0^{p}\,d\mathbf g.
\]
The H\"older inequality yields
\[
\left(\int_{R'_{1b}}\rho^{p}\,d\mathbf g\right)^{1/p}\left(\int_{R'_{1b}}\rho_0^{(p-1)q}\,d\mathbf g\right)^{1/q}\geq \int_{R'_{1b}}\rho\rho_0^{p-1}\,d\mathbf g\geq \int_{R'_{1b}}\rho_0^{p}\,d\mathbf g,
\]
or since $(p-1)q=p$,
\[
 \int_{R'_{1b}}\rho^{p}\,d\mathbf g \geq  \int_{R'_{1b}}\rho_0^{p}\,d\mathbf g.
\]
Taking infimum in the above inequality over all admissible $\rho$ we conclude that
\begin{equation}\label{moddd2}
M_p(f(\Gamma_0))\geq \int_{R'_{1b}}\rho_0^{p}\,d\mathbf g.
\end{equation}
Comparing \eqref{moddd1} and \eqref{moddd2} we see that
the function $\rho_0$ is extremal for the module $M_p(f(\Gamma_0))$. Now we can calculate the $p$-module as
\begin{eqnarray*}
M_p(f(\Gamma_0))&=&\int_{R'_{1b}}\rho_0^{p}\,d\mathbf g=\int_{S^1\setminus\mathcal Z} \int_1^b [\rho_0^{p}\circ f]\,J_f r^3dr\,dH^{3}\\
&=&\int_{S_1\setminus\mathcal Z} \int_1^b\frac{1}{\ell^p(\theta,\alpha)} \left(\frac{\|\dot{c}_{\theta\alpha}\|_0}{J_fr^3\sqrt{\cos\alpha}}\right)^{\frac{p}{p-1}}\,J_fr^3\,dr\,dH^3
\\
&=&\int_{-\pi/2}^{\pi/2}\int_0^{2\pi} \ell^{1-p}(\theta,\alpha)\,d\theta d\alpha.
\end{eqnarray*}
\end{proof}

\begin{example}
The $p$-module of $\Gamma=\Gamma(R_{a,b};S_a,S_b)$ 
\[
M_p(\Gamma_0)=M_p(\Gamma_0)=\begin{cases}
\frac{2\pi\sqrt{\pi}\Gamma\left(\frac{p}{4}+\frac{1}{2}\right)}{\Gamma\left(\frac{p}{4}+1\right)}\left(\frac{p-4}{p-1}\right)^{p-1}\left(b^{\frac{p-4}{p-1}}-1\right)^{1-p}, & \text{for $p\neq 4$},\\
\frac{\pi^2}{(\log b)^3},  & \text{for $p= 4$},
\end{cases}
\]
was calculated in \cite{KorReim87}. 

Let us now calculate the $p$-module of $f(\Gamma_0)$, where $f$ is a contact $C^1$-smooth orientation preserving map. If we try to create  a twisting
map similarly to Example~\ref{e7}   in the spherical coordinates  written as
$$G^{-1}\circ f\circ G\colon (\theta,\alpha, r)\to (\theta+\omega(r),\alpha,r),\quad \omega(1)=0,$$
with $G$ defined by~\eqref{eq:Heis_polar}, (i.e, the boundary
sphere $S_1$ remains unchanged while the  spheres $S_r$ rotate to the angle $\omega(r)$, $r\in (1,b]$), then the condition of horizontality for the curves $f(\phi_{\theta\alpha})$ is quite rigid, which
leads us to $\omega(r)\equiv 0$. Let us try to modify the twisting map by
$$G^{-1}\circ f\circ G\colon (\theta,\alpha, r)\to (\theta+\tan\alpha\log r +\omega_1(r),\alpha+\omega_2(r), r),\quad \omega_1(1)= \omega_2(1)=0.$$
Then the image $c_{\theta\alpha}(r)=f(\phi_{\theta\alpha}(r))$ is written in coordinates as
\begin{eqnarray*}
x_1(r)&=& r{\sqrt{|\cos(\alpha+\omega_2(r))|}}\cos(\theta+\omega_1(r)),\\
x_2(r)&=& r{\sqrt{|\cos(\alpha+\omega_2(r))|}}\sin(\theta+\omega_1(r)),\\
t(r)&=& r^2\sin(\alpha+\omega_2(r)).
\end{eqnarray*}
The horizontality condition $\dot{t}=2(\dot{x}_1x_2-\dot{x}_2x_1)$ is equivalent to $$\dot{\omega}_1=-\frac{1}{2}\dot{\omega}_2-\frac{1}{r}\tan(\alpha+\omega_2).$$
For example, 
\[
\omega_2(r)=r-1,\quad\omega_1=\frac{1-r}{2}-\int_1^r\frac{\tan(\alpha+s-1)}{s}ds.
\]
Then  $J_f=J_{G^{-1}\circ f\circ G}=1$ and
\[
\|\dot{c}_{\theta\alpha}\|_0=\frac{1}{2}\sqrt{\frac{4+r^2}{\cos(\alpha+r-1)}}.
\]
In Theorem~\ref{Rodin33} we calculate
\[
\ell(\theta,\alpha)=\ell(\alpha)=\int_1^b\left(\frac{4+r^2}{4\cos(\alpha+r-1)}\right)^{q/2}r^{3-q}(\cos{\alpha})^{\frac{1}{2}(1-q)}\,dr.
\]
and the $p$-module of $f(\Gamma_0)$ is
\[
M_p(f(\Gamma_0))=2\pi\int_{-\pi/2}^{\pi/2} \ell^{1-p}(\alpha)d\alpha.
\]
Unfortunately, the integral inequality which follows from the monotonicity of the module $M_p(f(\Gamma_0))\leq M_p(\Gamma_0)$ is quite difficult, and there is very little hope
to obtain simple inequalities as in Example~\ref{e7}.
\end{example}

\end{document}